\documentclass{amsart}

\usepackage[english]{babel}
\input xy
\xyoption{all}

\renewcommand{\leq}{\leqslant}
\renewcommand{\geq}{\geqslant}
\numberwithin{equation}{section}

\usepackage[utf8]{inputenc}
\usepackage{amsmath}
\usepackage{graphicx}
\usepackage{amssymb}
\usepackage{amsthm}
\usepackage{tikz-cd}
\usepackage{mathrsfs}
\usepackage{enumitem,verbatim}

\title[Collapsing Calabi-Yau manifolds and K\"ahler-Ricci flows]{On the collapsing of Calabi-Yau manifolds and K\"ahler-Ricci flows}

\author{Yang Li}
\address{Massachusetts Institute of Technology, 77 Massachusetts Avenue, Cambridge, MA 02139}
\email{yangmit@mit.edu}
\author{Valentino Tosatti}
\address{Courant Institute of Mathematical Sciences, 251 Mercer St, New York, NY 10012}
\email{tosatti@cims.nyu.edu}

\newtheorem{thm}{Theorem}[section]
\newtheorem{lem}[thm]{Lemma}
\newtheorem{prop}[thm]{Proposition}
\newtheorem{conj}[thm]{Conjecture}

\theoremstyle{definition}

\newtheorem{rmk}[thm]{Remark}

\newcommand{\C}{\mathbb{C}}

\newcommand{\ddbar}{i\partial\bar{\partial}}
\newcommand{\ov}[1]{\overline{#1}}
\newcommand{\ti}[1]{\tilde{#1}}
\newcommand{\ve}{\varepsilon}
\newcommand{\vp}{\varphi}
\newcommand{\de}{\partial}
\newcommand{\db}{\overline{\partial}}
\newcommand{\tr}[2]{\textrm{tr}_{#1} #2}

\DeclareMathOperator{\Vol}{Vol} 
\DeclareMathOperator{\Ric}{Ric}

\def\XXint#1#2#3{{\setbox0=\hbox{$#1{#2#3}{\int}$ }
		\vcenter{\hbox{$#2#3$ }}\kern-.6\wd0}}

\begin{document}
\begin{abstract}We study the collapsing of Calabi-Yau metrics and of K\"ahler-Ricci flows on fiber spaces where the base is smooth. We identify the collapsed Gromov-Hausdorff limit of the K\"ahler-Ricci flow when the divisorial part of the discriminant locus has simple normal crossings. In either setting, we also obtain an explicit bound for the real codimension $2$ Hausdorff measure of the Cheeger-Colding singular set, and identify a sufficient condition from birational geometry to understand the metric behavior of the limiting metric on the base.
\end{abstract}	
	\maketitle

\section{Introduction}

\medskip

In this paper we study the collapsing behavior of Ricci-flat K\"ahler metrics on Calabi-Yau manifolds, and of long-time solutions of the K\"ahler-Ricci flow. We first describe in detail these two setups, which have been much studied recently, and state the main open problems that we are interested in.

\subsection{Calabi-Yau}\label{setupcy}
$M^m$ is a projective Calabi-Yau manifold with $K_M\cong\mathcal{O}_M$, with a trivialization $\Omega$ of $K_M$, equipped with a holomorphic line bundle $\mathcal{L}$ which is semiample and with Iitaka dimension $n:=\kappa(\mathcal{L})$ that satisfies $0<n<m$. Then there is some $\ell$ sufficiently divisible such that the linear system $|\ell\mathcal{L}|$ defines a fiber space structure $f:M\to N$ (surjective holomorphic map with connected fibers) onto a normal projective variety $N^n$ with $0<n<m$. Let $D\subset N$ denote the closed subvariety given by the union of the singularities of $N$ together with the critical values of $f$ on $N^{\rm reg}$, and write $S=f^{-1}(D)$ and $D=D^{(1)}\cup D^{(2)}$ where $D^{(1)}$ is the union of all codimension $1$ irreducible components of $D$ and $\dim D^{(2)}\leq n-2$. The fibers $M_y=f^{-1}(y)$ for $y\in N\backslash D$ are Calabi-Yau $(m-n)$-folds. We will also denote by $N^\circ=N\backslash D, M^\circ=M\backslash S$.

Given a K\"ahler metric $\omega_N$ on $N$ (in the sense of analytic spaces \cite{Moi} if $N$ is not smooth) and a Ricci-flat K\"ahler metric $\omega_M$ on $M$, we are interested in the behavior of the Ricci-flat metrics $\omega(t)$ on $M$ cohomologous to $f^*\omega_N+e^{-t}\omega_N, t\geq 0$, in the limit as $t\to\infty$. To identify the limit, one solves \cite{ST, To0} the complex Monge-Amp\`ere equation on $N^\circ$
\begin{equation}\label{macy}
(\omega_N+\ddbar\vp)^n=f_*(\omega_M^m)\frac{\int_N\omega_N^n}{\int_M\omega_M^m},
\end{equation}
where $\omega_{\rm can}:=\omega_N+\ddbar\vp$ is a K\"ahler metric on $N^\circ$ and $\vp\in C^0(N)$ (for continuity, see \cite{CGZ,DZ,GGZ}).
After earlier work in \cite{To0,GTZ,TWY,HT,HT2}, it was very recently shown in \cite{HT3} that $\omega(t)\to f^*\omega_{\rm can}$ in $C^{\infty}_{\rm loc}(M^\circ,g_M)$.

In \cite{STZ} it was proved that the metric completion $(Z,d_Z)$ of $(N^\circ, \omega_{\rm can})$ is a compact metric space and that $(M,\omega(t))\to (Z,d_Z)$ in the Gromov-Hausdorff topology (see also \cite{GTZ2,TZ2} for earlier results in this direction). The following questions, raised in \cite{To2,To3,To4,GTZ2}, remain open in general:
\begin{conj}\label{ccy}
In the Calabi-Yau setup, the Gromov-Hausdorff limit is homeomorphic to $N$. Furthermore, $Z\backslash N^\circ$ has real Hausdorff codimension at least $2$ inside $(Z,d_Z)$.
\end{conj}
The homeomorphism statement was proved in \cite{STZ} when $N$ is smooth, and the full conjecture is known when $N$ is a curve \cite{GTZ2}, or when $M$ is hyperk\"ahler \cite{TZ2}, or when $N$ is smooth and $D^{(1)}$ has simple normal crossings \cite{GTZ3}.

We remark that the choice of path $f^*[\omega_N]+e^{-t}[\omega_N]$ in cohomology originates in \cite{GW}, and is quite analogous to what happens in the K\"ahler-Ricci flow setup below. Choosing a different path that approaches $f^*[\omega_N]$ in general results in a different behavior \cite[\S 4.4.4]{FT}, and the existing estimates mostly break down.

\subsection{K\"ahler-Ricci flow}\label{setupkrf} $M^m$ is a compact K\"ahler manifold with $K_M$ semiample and with Kodaira dimension $n=\kappa(M)$ that satisfies $0<n<m$. Let $f:M\to N$ be the Iitaka fibration of $M$, which is the fiber space determined by the linear system $|\ell K_M|$ with $\ell$ sufficiently divisible, and $N^n$ is a normal projective variety. Define $D,N^\circ, M^\circ$ as in Setup \ref{setupcy}, and again the fibers $M_y,y\in N^\circ,$ are Calabi-Yau $(m-n)$-folds. Let $\omega_N=\frac{1}{\ell}\omega_{\rm FS}|_N$ so that $f^*\omega_N$ is a smooth semipositive representative of $c_1(K_M)$.

Given a K\"ahler metric $\omega_M$ on $M$, consider the normalized K\"ahler-Ricci flow on $M$
$$\frac{\de}{\de t}\omega(t)=-\Ric(\omega(t))-\omega(t),\quad\omega(0)=\omega_M.$$
The flow exists for all $t\geq 0$ (see e.g. \cite{To}), and we are interested in the behavior as $t\to\infty$. Observe that the metric $\omega(t)$ is cohomologous to $(1-e^{-t})f^*\omega_N+e^{-t}\omega_M$.

In order to identify the limit of the flow, we fix a basis $\{s_i\}$ of $H^0(M,\ell K_M)$ which defines the map $f$, and obtain a smooth positive volume form $\mathcal{M}$ on $M$ by
$$\mathcal{M}=\left((-1)^{\frac{\ell m^2}{2}}\sum_i s_i\wedge\ov{s_i}\right)^\frac{1}{\ell}.$$
On then solves \cite{ST} the complex Monge-Amp\`ere equation on $N^\circ$
\begin{equation}\label{makrf}
(\omega_N+\ddbar\vp)^n=e^\vp f_*(\mathcal{M}),
\end{equation}
where $\omega_{\rm can}:=\omega_N+\ddbar\vp$ is a K\"ahler metric on $N^\circ$ and $\vp\in C^0(N)$.
After earlier work in \cite{ST,FZ,TWY,HT,FL}, it was recently shown in \cite{CL} that $\omega(t)\to f^*\omega_{\rm can}$ in $C^{\alpha}_{\rm loc}(M^\circ)$ as $t\to\infty$, for any $0<\alpha<1$.

 Furthermore, in \cite{JS} it is shown that $\mathrm{diam}(M,\omega(t))\leq C$, for all $t\geq 0$, and \cite{STZ} shows that the metric completion $(Z,d_Z)$ of $(N^\circ, \omega_{\rm can})$ is a compact metric space, which is homeomorphic to $N$ when this is smooth.
\begin{conj}\label{ckrf}
In the K\"ahler-Ricci flow setup, $(M,\omega(t))\to (Z,d_Z)$ in the Gromov-Hausdorff topology. Furthermore, $Z$ is homeomorphic to $N$ and $Z\backslash N^\circ$ has real Hausdorff codimension at least $2$ inside $(Z,d_Z)$.
\end{conj}

The Gromov-Hausdorff convergence is known when $N$ is a curve and the generic fibers of $f$ are tori \cite{STZ}.

\subsection{Our results} We can now state our results.
In either the Calabi-Yau setup \ref{setupcy} or the K\"ahler-Ricci flow setup \ref{setupkrf}, assume that $N$ is smooth and let $(Z,d_Z)$ be the metric completion of $(N^\circ, \omega_{\rm can})$. Thanks to \cite{STZ} this is a compact metric space homeomorphic to $N$, and in the Calabi-Yau setup \ref{setupcy} it is the Gromov-Hausdorff limit of $(M,\omega(t))$ as $t\to\infty$.

Let $\mathcal{S}\subset Z$ be the singular set in the sense of Cheeger-Colding \cite{CC}, namely the set of all $x\in Z$ such that there is some tangent cone to $(Z,d_Z)$ at $x$ which is not isometric to $\mathbb{R}^{2n}$. We always have $\mathcal{S}\subset Z\backslash N^\circ$, but this inclusion is strict in general (see Remark \ref{rkincl}). Our first result, proved in section \ref{sect2}, is an explicit Hausdorff measure bound for $\mathcal{S}$:

\begin{thm}\label{cy_smoothbase}
In either the Calabi-Yau setup \ref{setupcy} or the K\"ahler-Ricci flow setup \ref{setupkrf}, assume that $N$ is smooth and $[\omega_N]\in H^2(N,\mathbb{Q})$, and let $\mathcal{H}^{2n-2}$ be the real $(2n-2)$-dimensional Hausdorff measure of the limit metric $d_Z$ on $N$. Then the Cheeger-Colding singular set $\mathcal{S}$ satisfies
\begin{equation}\label{volume}
\mathcal{H}^{2n-2}(\mathcal{S})\leq C_n\int_D\omega_N^{n-1},
\end{equation}
where $C_n$ is a dimensional constant.
\end{thm}

This estimate would be expected if the Hausdorff measure could indeed be computed cohomologically, as in the case when the limiting metric $d_Z$ is smooth. However, even the best pointwise estimate \eqref{desired} below cannot by itself imply this measure bound, and one needs instead to appeal to the deep work of Liu-Sz\'ekelyhidi \cite{LS} on Gromov-Hausdorff noncollapsed limits of polarized K\"ahler manifolds with Ricci bounded below. The idea is to use standard approximations $\omega_j$ of $\omega_{\rm can}$ and study the singularities of the closed positive current $\Ric$ on $N$ which is the weak limit of the Ricci curvature of $\omega_j$. The results in \cite{LS} characterize $\mathcal{S}$ as the set of points in $N$ where the limiting Ricci curvature current has positive Lelong number.  At almost all points $x\in\mathcal{S}$ the tangent cone is $\mathbb{R}^{2n-2}\times\mathbb{C}_{\theta(x)}$, where $\mathbb{C}_{\theta(x)}$ denotes the standard conical metric in $\mathbb{C}$ with cone angle $2\pi\theta(x)$ at $0$. We are able to relate $\theta(x)$ to the Lelong number of the limiting Ricci current at $x$, which can be estimated thanks to the asymptotics of the volume form $\omega_{\rm can}^n$ proved in \cite{GTZ3}, and we then estimate the Hausdorff measure using the scalar curvature integral.

Our next result deals with the general case when $N$ is allowed to be singular. We let $\pi:\ti{N}\to N$ be a resolution of singularities with $\ti{N}$ smooth and $\pi^{-1}(D)=:E=\cup_i E_i$ a simple normal crossings divisor. In \cite{TZ2} the second-named author and Zhang conjectured that we can find such a resolution such that on $\ti{N}\backslash E$ we have the estimate
\begin{equation}\label{desired}
\pi^*\omega_{\rm can}\leq C\left(1-\sum_i \log|s_i|_{h_i}\right)^C \omega_{\rm cone},
\end{equation}
where $s_i$ is a defining section of $E_i$, $h_i$ is a Hermitian metric on $\mathcal{O}(E_i)$, and $\omega_{\rm cone}$ is a K\"ahler metric on $\ti{N}\backslash E$ with conical singularities along $E$ with cone angles $2\pi \alpha_i (0<\alpha_i\leq 1)$ along $E_i$ (we are assuming here without loss that $|s_i|_{h_i}\leq 1$ on $\ti{N}$, so that the multiplying factor on the RHS of \eqref{desired} is bounded away from zero). Building upon \cite{GTZ2}, it was proved in \cite{TZ2} that the estimate \eqref{desired} would imply the Hausdorff dimension bound in Conjecture \ref{ccy} in full generality (this was slightly relaxed in \cite{Br} to allow for arbitrary small extra poles along $E$ on the RHS of \eqref{desired}). The conjectured estimate \eqref{desired} was proved in \cite{GTZ2} when $\dim N=1$, in \cite{TZ2} when $M$ is hyperk\"ahler, and in \cite{GTZ3} when $N$ is smooth and $D^{(1)}$ has simple normal crossings, but it remains open in general. Our next result identifies an algebro-geometric condition which is sufficient to prove \eqref{desired}, and which comes from the canonical bundle formula in birational geometry: roughly speaking, to any resolution $\pi:\ti{N}\to N$ as above (together with a resolution of the pullback of $f:M\to N$ over $\ti{N}$), we associate a $\mathbb{Q}$-divisor $\Xi_{\ti{N}}$ on $\ti{N}$, which is functorial in the sense that passing to a higher model gives the pullback divisor. In  section \ref{sectbir} we then show:
\begin{thm}\label{cy_canonical}
In either the Calabi-Yau setup in \ref{setupcy} or the K\"ahler-Ricci flow setup in \ref{setupkrf}, the conjectured estimate \eqref{desired} holds provided that there exists a resolution $\pi:\ti{N}\to N$ as above such that $\Xi_{\ti{N}}$  is $\pi$-ample.
\end{thm}

To prove this result, we refine the arguments in \cite{GTZ3} and identify a key divisor $\Xi_{\ti{N}}$ on the resolution $\ti{N}$ with the property that whenever $\Xi_{\ti{N}}$ is $\pi$-ample then the desired estimate \eqref{desired} can be shown to hold. We then describe $\Xi_{\ti{N}}$ explicitly using the canonical bundle formula and the recent results of Kim \cite{Ki}.

Our last result, proved in section \ref{sectkrf},  settles the Gromov-Hausdorff convergence in Conjecture \ref{ckrf} under a log smoothness assumption:

\begin{thm}\label{krf_gh}
Assume the K\"ahler-Ricci flow setup \ref{setupkrf}, and suppose that $N$ is smooth and $D^{(1)}$ is a simple normal crossings divisor. Then $(M,\omega(t))$ has converges in the Gromov-Hausdorff topology to the metric completion of $(N^\circ,\omega_{\rm can})$, which by \cite{STZ} is homeomorphic to $N$.
\end{thm}

This is the first time that this conjectural Gromov-Hausdorff convergence is proved for base spaces $N$ of dimension greater than $1$ (as mentioned earlier, it was only previously known under the stringent assuptions that $N$ is a curve and the generic fibers of $f$ are tori \cite{STZ}), and our assumptions that $N$ is smooth and $D^{(1)}$ is snc can be thought of as generic, since for example they are Zariski open in families. One major difference between our result and those in \cite{STZ} is that when $N$ is a curve then $D$ is a discrete set, and in this case it follows from \cite{GTZ2} that a small tubular neighborhood of $D$ has very small $\omega_{\rm can}$-diameter (which is used in \cite{STZ}), while this is clearly false when $N$ is higher-dimensional. To prove our result, we make use of the fact proved in \cite{STZ} that $(N^\circ,\omega_{\rm can})$ is ``almost-convex'', and the bulk of our work lies in establishing an analogous statement for $(M^\circ,\omega(t))$, uniformly in $t$. This in turn requires new ideas, combining results of Perelman \cite{Pe} with estimate \eqref{desired} to analyze the behavior of $\mathcal{L}$-geodesics which have endpoints away from $S=f^{-1}(D)$ but which may venture quite close to $S$.

\subsection{Acknowledgments} This research was conducted during the period when the first-named author served as a Clay Research Fellow. The second-named author was partially supported by NSF grants DMS-1903147 and DMS-2231783, and part of this work was carried out during his visit to the Department of Mathematics and the Center for Mathematical Sciences and Applications at Harvard University, which he would like to thank for the hospitality. The second-named author would like to thank Henri Guenancia, Dano Kim, Man-Chun Lee and Yuguang Zhang for discussions. We are also grateful to the referee for useful comments.

\section{Hausdorff measure bound for the singular set}\label{sect2}
In this section we prove Theorem \ref{cy_smoothbase}. Throughout this section we assume that $N$ is smooth and furthermore that $[\omega_N]\in H^2(N,\mathbb{Q})$ (this is automatic in the K\"ahler-Ricci flow setup). By the Kodaira embedding theorem, this means that $[\omega_N]=c_1(L)$ where $L\to N$ is an ample line bundle, which is needed to apply the results of \cite{LS}.

\subsection{The approximation procedure}
For ease of notation, in the Calabi-Yau setup \ref{setupcy} we denote by
\begin{equation}
\mathcal{M}=\frac{\int_N\omega_N^n}{\int_M\omega_M^m}\omega_M^m,
\end{equation}
so that in both setups \ref{setupcy} and \ref{setupkrf} we can write the Monge-Amp\`ere equations \eqref{macy} and \eqref{makrf} on $N^\circ=N\backslash D$ as
\begin{equation}\label{ma0}
(\omega_N+\ddbar\vp)^n=e^{\lambda\vp} f_*(\mathcal{M}),
\end{equation}
where $\lambda=0$ in the Calabi-Yau setup and $\lambda=1$ in the K\"ahler-Ricci flow setup.

As shown in \cite[Prop.3.1]{GTZ3}, $\omega_{\rm can}$ extends to a smooth K\"ahler metric across $D^{(2)}$, so without loss we may assume that $D=D^{(1)}$ is a divisor (not necessarily with simple normal crossings).

Recall from the introduction that we have defined $(Z,d_Z)$ to be the metric completion of $(N^\circ,\omega_{\rm can})$, which by \cite{STZ} is a compact metric space homeomorphic to $N$ (using here that $N$ is smooth).

We then define a smooth positive function $\mathcal{F}$ on $N^\circ$  by
\begin{equation}\label{deff}\frac{f_*(\mathcal{M})}{\omega_N^n}=\mathcal{F},
\end{equation}
which as shown in \cite{ST} (see also \cite{To0}, \cite[Prop.5.9]{To}) satisfies
\begin{equation}\label{ineq}
\lambda\omega_N+\Ric(\omega_N)-\ddbar\log \mathcal{F}=\omega_{\rm WP}\geq 0,
\end{equation}
where $\omega_{\rm WP}$ is a semipositive form of Weil-Petersson type.
As shown e.g. in \cite[Lemma 4.1]{GTZ3} we have
\begin{equation}\label{asino2}
\mathcal{F}\geq C^{-1},
\end{equation}
on $N^\circ$, so by Grauert-Remmert \cite{GR}, $-\log\mathcal{F}$ extends to a quasi-psh function on $N$, still denoted by $-\log\mathcal{F}$, which satisfies \eqref{ineq} weakly on $N$, and in general it may have values $-\infty$ along $D$. Also, as shown in \cite[Proposition 3.2]{ST} (see also \cite[Lemma 2.1]{EGZ}, \cite[Proposition 4]{DF}) there is $p>1$ such that $\mathcal{F}\in L^p(N,\omega_N^n)$.

We can then apply Demailly's regularization \cite[Theorem 9.1]{De} and obtain a sequence of smooth functions $v_j$ on $N$ which decrease pointwise to $-\log\mathcal{F}$ as $j\to\infty$, and satisfy $v_j\leq C$ for all $j$ and
\begin{equation}\label{ricc}
\lambda\omega_N+\Ric(\omega_N)+\ddbar v_j\geq -C\omega_N,
\end{equation}
for all $j$. The lower bound here cannot be taken arbitrarily small in general because $-\log\mathcal{F}$ can have positive Lelong number at points in $D$. Furthermore, as the construction of $v_j$ shows, we have $v_j\to -\log\mathcal{F}$ smoothly on every given compact subset of $N^\circ$.

By Yau's Theorem \cite{Ya} (and also Aubin \cite{Au} when $\lambda=1$) we can find K\"ahler metrics $\omega_j=\omega_N+\ddbar\vp_j$ on $N$ which satisfy
\begin{equation}\label{ma}
\omega_j^n=(\omega_N+\ddbar\vp_j)^n=c_je^{\lambda\vp_j-v_j}\omega_N^n,
\end{equation}
where
\[
\left\{
                \begin{aligned}
                  &c_j=1,\quad \text{if}\ \lambda=1,\\
                  &c_j=\frac{\int_N\omega_N^n}{\int_N e^{-v_j}\omega_N^n},\quad \text{if}\ \lambda=0,
                \end{aligned}
              \right.
\]
and by construction we have
$$\int_Ne^{-v_j}\omega_N^n\to \int_N\mathcal{F}\omega_N^n=\int_Nf_*(\mathcal{M})=\int_M\mathcal{M},$$
which is strictly positive when $\lambda=1$ and equals $\int_N\omega_N^n$ when $\lambda=0$, so that in particular $c_j\to 1$ as $j\to\infty$. When $\lambda=0$ we also normalize $\vp_j$ by $\sup_N\vp_j=0$. When $\lambda=1$ we can apply the maximum principle to get
$$\sup_N\vp_j\leq \sup_N v_j\leq C,$$
independent of $j$. Also, in this case we have $\omega_j^n\leq e^{\sup_N\vp_j-v_j}\omega_N^n$ and integrating this gives
$$e^{\sup_N\vp_j}\geq \frac{\int_N\omega_N^n}{\int_Ne^{-v_j}\omega_N^n}\geq C^{-1}>0.$$
We thus conclude that when $\lambda=1$ we have $|\sup_N\vp_j|\leq C$ independent of $j$.
 Then, for $\lambda=0,1$, since as we said $\int_N\mathcal{F}^p\omega_N^n\leq C$ for some $p>1$, it follows that
\begin{equation}\label{estimatio}
c_j^p\int_N e^{p(\lambda\vp_j-v_j)}\omega_N^n\leq C,
\end{equation}
for all $j$, and so Ko\l odziej \cite{Ko} gives us
\begin{equation}\label{est}
\sup_N|\vp_j|\leq C,
\end{equation}
for all $j$. When $\lambda=0$ we have that $c_je^{-v_j}\to \mathcal{F}$ in $L^1(N,\omega_N^n)$, and so Ko\l odziej's stability theorem \cite{Ko2} gives us that
\begin{equation}\label{stab}
\|\vp_j-\vp\|_{L^\infty(N)}\to 0,
\end{equation}
where $\vp$ is as in \eqref{ma0}. For $\lambda=1$ we can still conclude that
\begin{equation}\label{stab2}
\|\vp_j-\vp\|_{L^1(N,\omega_N^n)}\to 0,
\end{equation}
by arguing as in \cite[Theorem 4.5]{BG}.

The following proposition is contained in \cite[Lemma 2.2]{FGS} and \cite[Prop. 2.3]{STZ}, but we include the proof for convenience:

\begin{prop}\label{gh} The approximating metrics $\omega_j$ on $N$ satisfy
\begin{equation}\label{1}
\Ric(\omega_j)\geq -C\omega_j,
\end{equation}
\begin{equation}\label{2}
\mathrm{diam}(N,\omega_j)\leq C,
\end{equation}
\begin{equation}\label{3}
\mathrm{Vol}_{\omega_j}B_{\omega_j}(x,r)\geq C^{-1}r^{2n},
\end{equation}
for all $0<r\leq 1, x\in N$ and $j$. The distance functions $d_{g_j}$ defined by $(N,\omega_j)$ satisfy
\begin{equation}\label{4}
d_{g_j}(p,q)\leq Cd_{g_N}(p,q)^\alpha,
\end{equation}
for some $C,\alpha>0$ and for all $p,q\in N,j\geq 0$.
We also have $\omega_j\to\omega_{\rm can}$ locally smoothly on $N^\circ$, and $(N,\omega_j)\to (Z,d_Z)$ in the Gromov-Hausdorff topology.
\end{prop}
\begin{proof}
From \eqref{ma} and \eqref{ricc} we get
    \begin{equation}\label{st}
\Ric(\omega_j)=\Ric(\omega_N)+\ddbar v_j-\lambda\ddbar\vp_j\geq -C\omega_N-\lambda\omega_j,
\end{equation}
so to prove \eqref{1} it suffices to show that
\begin{equation}\label{sl}
\omega_j\geq C^{-1}\omega_N.
\end{equation}
This follows from the usual Schwarz Lemma argument: the Chern-Lu inequality gives
$$\Delta_{\omega_j}\log\tr{\omega_j}{\omega_N}\geq \frac{1}{\tr{\omega_j}{\omega_N}}\left(g_j^{k\ov{\ell}}g_j^{p\ov{q}}\Ric(\omega_j)_{k\ov{q}}(g_N)_{p\ov{\ell}}-g_j^{k\ov{\ell}}g_j^{p\ov{q}}({\rm Rm}(\omega_N))_{k\ov{\ell}p\ov{q}}\right)\geq -C\tr{\omega_j}{\omega_N}-\lambda,$$
using \eqref{st}, and so taking $A$ large enough but uniform we have
$$\Delta_{\omega_j}\left(\log\tr{\omega_j}{\omega_N}-A\vp_j\right)\geq \tr{\omega_j}{\omega_N}-C,$$
and so the maximum principle and \eqref{est} give
$$\tr{\omega_j}{\omega_N}\leq C,$$
which proves \eqref{sl}.

Next, applying \cite[Theorem 1.1]{FGS} directly proves \eqref{2}, and then Bishop-Gromov volume comparison gives us \eqref{3}. The H\"older estimate \eqref{4} for the distance function of $(N,\omega_j)$ follows from Ko\l odziej's uniform H\"older bound \cite{Ko3}
$$\|\vp_j\|_{C^{2\alpha}(N,\omega_N)}\leq C,$$
for some $\alpha>0$, together with the first-named author's argument in \cite[Theorem 4.1]{Li} that deduces \eqref{4} from this.

To prove locally smooth convergence, observe that for every given $K\Subset N^\circ$ we have in particular that $\sup_K (c_je^{\lambda\vp_j-v_j})\leq C_K$ for all $j$, and combining this with \eqref{ma} and \eqref{sl} we see that on $K$ we have
$$C_K^{-1}\omega_N\leq \omega_j\leq C_K\omega_N,$$
for all $j$, and by now standard local higher order estimates for \eqref{ma} (on a slightly larger open set) give us uniform estimates $\|\omega_j\|_{C^k(K,\omega_N)}\leq C_{K,k}$ for all $j,k$. Thanks to \eqref{stab} and \eqref{stab2}, this gives that $\vp_j\to \vp$ smoothly on $K$, and so $\omega_j\to \omega_{\rm can}$ locally smoothly on $N^\circ$.

Lastly, the Gromov-Hausdorff convergence of $(N,\omega_j)$ to $(Z,d_Z)$ follows by combining the arguments in \cite[Proof of Prop.2.3 Step 3]{STZ} and \cite[Proof of Prop.2.2 (3)]{STZ}.
\end{proof}

\subsection{The measure bound \eqref{volume}}
In this section we still have as standing assumption that $N$ is smooth. Define $\mathcal{S}\subset Z$ as the singular set in the sense of Cheeger-Colding \cite{CC}, namely the set of all $x\in Z$ such that there is some tangent cone at $x$ which is not isometric to $\mathbb{R}^{2n}$, and for $0\leq k\leq 2n$ let
$\mathcal{S}^{k}$ be the set of all $x\in Z$ such that no tangent cone at $x$ splits off an $\mathbb{R}^{k+1}$ factor. Then, thanks to Proposition \ref{gh},  by \cite{CC} we have $\mathcal{S}=\mathcal{S}^{2n-2}$ and $\dim_{\mathcal{H}}\mathcal{S}^k\leq k$. In particular, if we define
$$\Sigma=\mathcal{S}\backslash \mathcal{S}^{2n-3}=\mathcal{S}^{2n-2}\backslash \mathcal{S}^{2n-3},$$
then for every point $x\in\Sigma$ there is some tangent cone at $x$ which splits off $\mathbb{R}^{2n-2}$, and $\dim_{\mathcal{H}}(\mathcal{S}\backslash\Sigma)\leq 2n-3$. Furthermore, thanks to \cite{CJN} up to removing a subset of $\Sigma$ with zero $2n-2$ dimensional Hausdorff measure (which we will do without changing notation), the tangent cone at any $x\in \Sigma$ is unique and isometric to $\mathbb{R}^{2n-2}\times\mathbb{C}_{\theta(x)}$, where $0<\theta(x)<2\pi$ denotes the cone angle at $x$. The function $\theta(x)$ can be interpreted as the monotone limit of the volume ratio at $x$, $\theta(x)=\lim_{r\downarrow 0}\frac{\Vol B(x,r)}{\omega_{2n}r^{2n}}$, whence it is upper-semicontinuous.

As in \cite{CC}, for any $\ve>0$ we define $\mathcal{R}_\ve$ as the set of all points $x\in Z$ such that the Gromov-Hausdorff distance between $B(x,r)$ and the $r$-ball in $\mathbb{R}^{2n}$ is less than $\ve r$ for all sufficiently small $r>0$. Then their complements
\begin{equation}\label{se}
\mathcal{S}_\ve=Z\backslash\mathcal{R}_\ve,
\end{equation}
 are closed subsets and $\mathcal{S}=\bigcup_{\ve>0}\mathcal{S}_\ve$.

Recall in \cite[Prop. 2.3]{STZ} it is shown the inclusion $N^\circ\hookrightarrow N$ extends to a homeomorphism $F:Z\to N$, which maps bijectively $Z\backslash\iota(N^\circ)$ onto $D\subset N$ (here $\iota:N^\circ\hookrightarrow Z$ is the canonical inclusion). We will use $F$ to identify $Z$ with $N$, suppressing $F$ from the notation, so for example $d_Z$ will be a distance function on $N$, etc. It is important to note that the Hausdorff measures and dimensions that we will use on $N$ are those of $d_Z$ (and not those of a smooth metric on $N$), unless otherwise specified.

\begin{rmk}\label{rkincl}
Since $\omega_{\rm can}$ is a smooth K\"ahler metric on $N^\circ$ and $\omega_j\to \omega_{\rm can}$ locally smoothly there, it follows that $\mathcal{S}\subset D$. This inclusion is strict in general, as can be seen for example in the case when $f:M\to N$ is an elliptic fibration of $K3$ surfaces with $24$ singular fibers of type $I_1$, which is the setup considered by Gross-Wilson \cite{GW}:  in this case $D$ is a finite set of points in $N\cong\mathbb{P}^1$ and from their work it follows that the metric $\omega_{\rm can}$ has tangent cone $\mathbb{C}$ at all points of $D$ (indeed, $\omega_{\rm can}$ has an explicit asymptotic behavior at points in $D$, see e.g. \cite[Table 1]{He}), so in this case $\mathcal{S}$ is empty even though the metric is not smooth at the points in $D$. This was extended in \cite{GTZ2} to arbitrary elliptically fibered $K3$ surfaces, and the tangent cone of $\omega_{\rm can}$ at any point $p\in D$ can be precisely determined from the Kodaira type of the singular fiber $f^{-1}(p)$, see \cite[Table 1]{He} (in particular, the tangent cone is $\mathbb{C}$ if and only if the singular fiber is of type $I_b, b\geq 0$).
\end{rmk}

Thanks to \cite[Proposition 4.1]{LS}, there is a weak limit $\mathrm{Ric}$ of $\mathrm{Ric}(\omega_j)$, which is a closed $(1,1)$-current on $N$, smooth on $N^\circ$, which locally differs from a positive current by $\ddbar$ of a continuous function, hence its Lelong numbers are well-defined. They also show that $\mathcal{S}$ is equal to the set of points $x\in N$ where $\nu(\mathrm{Ric},x)>0$. By \cite[Theorem 4.1]{LS} this is an at most countable union of closed analytic subsets of $N$, contained in the discriminant locus $D$, and so in particular the number of divisorial components of $\mathcal{S}$ is finite. Passing to the limit in \eqref{1} on $N^\circ$ shows that
\begin{equation}\label{riclower}\Ric\geq -C\omega_{\rm can},\end{equation}
holds pointwise on  $N^\circ$ and  weakly on all of $N$.

On the other hand, differentiating \eqref{ma} gives
\begin{equation}\label{a1}
\Ric(\omega_j)=\Ric(\omega_N)+\ddbar(v_j-\lambda \vp_j),
\end{equation}
where $\vp_j\to \vp$ uniformly on $N$ and $v_j$ decreases pointwise to $-\log\mathcal{F}$, and thus from the construction in \cite{LS} we see that we have
\begin{equation}\label{a2}
\Ric=\Ric(\omega_N)-\ddbar(\lambda\vp+\log\mathcal{F}),
\end{equation}
as currents on $N$, where recall that $\vp\in C^0(N)\cap C^\infty(N^\circ)$.

We write $D=\bigcup_i D_i$ for the decomposition into irreducible components (which are divisors, since as mentioned earlier we are assuming without loss that $D=D^{(1)}$, as  $\omega_{\rm can}$ extends smoothly across $D^{(2)}$ by \cite[Prop.3.1]{GTZ3}), and consider a composition of smooth blowups $\pi:\ti{N}\to N$ such that $E=\pi^{-1}(D)$ has simple normal crossings. Write  $\ti{D}$ for the proper transform of $D$ and $E=\pi^{-1}(D)=\ti{D}\cup F$ where $F$ is $\pi$-exceptional. Then $\pi^*\Ric$ has a Siu decomposition \cite{Siu}
\begin{equation}\label{lollobu}
\pi^*\Ric=\sum_i \nu(\Ric,D_i)[\ti{D}_i]+\sum_i \nu(\pi^*\Ric,F_i)[F_i]+\widetilde{\Ric}_{\rm sm},
\end{equation}
where $\nu(\Ric,D_i)=\nu(\pi^*\Ric,\ti{D}_i)$ and $\nu(\pi^*\Ric,F_i)$ are the generic Lelong numbers (which may be zero), and $\widetilde{\Ric}_{\rm sm}$ is a closed $(1,1)$-current on $\ti{N}$, smooth on $\ti{N}\backslash E$, which satisfies
\begin{equation}\label{riclower2}\widetilde{\Ric}_{\rm sm}\geq -C\pi^*\omega_{\rm can},\end{equation}
weakly on $\ti{N}$, and whose generic Lelong number along the $D_i$'s and $F_i$'s vanish. In fact we can say a lot more:

\begin{lem}\label{loglog}
For any $x\in E$ there is a neighborhood $U$ of $x$ in $\ti{N}$ and a constant $C_U$ so that on $U$ we can write
$$\widetilde{\Ric}_{\rm sm}=\ddbar\psi,$$
where $\psi$ satisfies
\begin{equation}\label{goalie}-C_U\log(-\log d_{g_{\ti{N}}}(z,E))\leq \psi(z)\leq C_U,
\end{equation}
for all $z\in U\backslash E$.
\end{lem}
\begin{proof}
From \eqref{a2} we have
$$\pi^*\Ric=\pi^*\Ric(\omega_N)-\ddbar(\lambda\pi^*\vp+\log\pi^*\mathcal{F}),$$
and since $\vp\in C^0(N)\cap C^\infty(N^\circ)$, it suffices to understand the singularities of $\pi^*\mathcal{F}$ along $E$.

This is a consequence of results proved in \cite{GTZ3} (generalizing earlier results in \cite{GTZ2} when $n=1$) as follows. Define $J_\pi\geq 0$ by
$$\pi^*\omega_N^n=J_\pi \omega_{\ti{N}}^n.$$
Choosing defining sections $s_{\ti{D}_i},s_{F_i}$ and metrics $h_{\ti{D}_i},h_{F_i}$ for the line bundles corresponding to the irreducible components of $\ti{D}$ and $F$, we have that $J_\pi$ is comparable to $\prod_j|s_{F_j}|_{h_{F_j}}^{2b_j}$ with $b_j\in\mathbb{N}_{>0}$ (recall that $N$ is smooth).
Then \cite[Theorems 2.3, 7.1 and Lemma 4.1]{GTZ3} (also \cite[Rmk 1.6]{Ki}) give on $\ti{N}\backslash E$
$$C^{-1}\frac{\prod_{j} |s_{F_j}|_{h_{F_j}}^{2\beta_j}}{\prod_{i}|s_{\ti{D}_i}|_{h_{\ti{D}_i}}^{2(1-\gamma_i)}}\leq J_\pi\pi^*\mathcal{F}\leq C\frac{\prod_{j} |s_{F_j}|_{h_{F_j}}^{2\beta_j}}{\prod_{i}|s_{\ti{D}_i}|_{h_{\ti{D}_i}}^{2(1-\gamma_i)}}(-\log d_{g_{\ti{N}}}(\cdot, E))^C,$$
where $\beta_j\in\mathbb{R}$ and $0<\gamma_i\leq 1$, and we must also have $b_j\geq \beta_j$. Thus
$$C^{-1}\frac{1}{\prod_{j} |s_{F_j}|_{h_{F_j}}^{2(b_j-\beta_j)}\prod_{i}|s_{\ti{D}_i}|_{h_{\ti{D}_i}}^{2(1-\gamma_i)}}\leq \pi^*\mathcal{F}\leq C\frac{(-\log d_{g_{\ti{N}}}(\cdot, E))^C}{\prod_{j} |s_{F_j}|_{h_{F_j}}^{2(b_j-\beta_j)}\prod_{i}|s_{\ti{D}_i}|_{h_{\ti{D}_i}}^{2(1-\gamma_i)}},$$
which shows that we can take $\psi$ equal to the sum of a local continuous function plus
$$-\log\left(\pi^*\mathcal{F}\prod_{j} |s_{F_j}|_{h_{F_j}}^{2(b_j-\beta_j)}\prod_{i}|s_{\ti{D}_i}|_{h_{\ti{D}_i}}^{2(1-\gamma_i)}\right),$$
and it satisfies \eqref{goalie} as claimed.
\end{proof}

Next, recall that $\omega_{\rm can}$ has continuous potentials on $N$, hence the Bedford-Taylor products $\omega_{\rm can}^{j}, 2\leq j\leq n,$ are well-defined closed positive $(j,j)$-currents on $N$ by \cite{BT}, whose cohomology class agrees with $[\omega_N^j]$ by \cite[Corollary 9.3]{Dem}. Also, since the unbounded locus of the local potentials of $\Ric+C\omega_{\rm can}$ is contained in $D$, which has $g_N$-Hausdorff dimension at most $2n-2$, the wedge product
$$(\Ric+C\omega_{\rm can})\wedge\omega_{\rm can}^{n-1}$$
is a well-defined positive Borel measure on $N$ by \cite[Theorem 2.5]{Dem}, whose total mass equals
$$\int_N (\Ric(\omega_N)+C\omega_N)\wedge\omega_N^{n-1},$$
again by \cite[Corollary 9.3]{Dem}. Furthermore, from \eqref{a1}, \eqref{a2}, and since $\vp_j\to \vp$ uniformly and $v_j$ decreases to $-\log\mathcal{F}$, \cite[Proposition 2.9]{Dem} shows that
\begin{equation}\label{lollo}
(\Ric(\omega_j)+C\omega_j)\wedge\omega_j^{n-1}\to (\Ric+C\omega_{\rm can})\wedge\omega_{\rm can}^{n-1},
\end{equation}
weakly as measures on $N$.

The pullbacks currents $\pi^*\Ric$ and $\pi^*\omega_{\rm can}$ are defined in the usual way (pulling back $\de\db$-potentials), the measure
$\pi^*(\Ric+C\omega_{\rm can})\wedge(\pi^*\omega_{\rm can})^{n-1}$ on $\ti{N}$ is defined as above using \cite[Theorem 2.5]{Dem}, and since $\pi$ is a modification one easily checks that we have
\begin{equation}\label{lolobu}
\pi_*(\pi^*(\Ric+C\omega_{\rm can})\wedge(\pi^*\omega_{\rm can})^{n-1})=(\Ric+C\omega_{\rm can})\wedge\omega_{\rm can}^{n-1}.
\end{equation}
Using \eqref{lollobu}, on $\ti{N}$ we have
\begin{equation}\label{lolo}\begin{split}
&\pi^*(\Ric+C\omega_{\rm can})\wedge(\pi^*\omega_{\rm can})^{n-1}\\
&=\sum_i \nu(\Ric,D_i)[\ti{D}_i]\wedge(\pi^*\omega_{\rm can})^{n-1}+\sum_i \nu(\pi^*\Ric,F_i)[F_i]\wedge(\pi^*\omega_{\rm can})^{n-1}\\
&\ \ \ \ +(\widetilde{\Ric}_{\rm sm}+C\pi^*\omega_{\rm can})\wedge(\pi^*\omega_{\rm can})^{n-1}\\
&=\sum_i \nu(\Ric,D_i)[\ti{D}_i]\wedge(\pi^*\omega_{\rm can})^{n-1}+(\widetilde{\Ric}_{\rm sm}+C\pi^*\omega_{\rm can})\wedge(\pi^*\omega_{\rm can})^{n-1},
\end{split}\end{equation}
because each term $[F_i]\wedge(\pi^*\omega_{\rm can})^{n-1}$ vanishes as $F_i$ is $\pi$-exceptional.

Let now $U_r\subset N$ be the $r$-neighborhood of $D$ with respect to $\omega_N$. We then have the following claim:

\begin{lem}\label{cuz}
For any continuous nonnegative function $h$ on $N$ we have
\[
\lim_{r\to 0}\int_{U_r} h(\Ric+C\omega_{\rm can})\wedge \omega_{\rm can}^{n-1} = \sum_i \nu(\Ric,D_i)\int_{D_i}h\omega_{\rm can}^{n-1}.
\]
\end{lem}
\begin{proof}
Let $\ti{U}_r=\pi^{-1}(U_r)$ (a shrinking family of open neighborhoods of $E$) and $\ti{h}=\pi^*h$, then using \eqref{lolo} we have
\[\begin{split}
&\lim_{r\to 0}\int_{U_r} h(\Ric+C\omega_{\rm can})\wedge \omega_{\rm can}^{n-1}= \lim_{r\to 0}\int_{\ti{U}_r} \ti{h}\pi^*(\Ric+C\omega_{\rm can})\wedge(\pi^*\omega_{\rm can})^{n-1}\\
&=\sum_i\nu(\Ric,D_i)\int_{\ti{D}_i}\ti{h}\pi^*\omega_{\rm can}^{n-1}+\lim_{r\to 0}\int_{\ti{U}_r}\ti{h} (\widetilde{\Ric}_{\rm sm}+C\pi^*\omega_{\rm can})\wedge(\pi^*\omega_{\rm can})^{n-1}\\
&=\sum_i\nu(\Ric,D_i)\int_{D_i}h\omega_{\rm can}^{n-1}+\lim_{r\to 0}\int_{\ti{U}_r} \ti{h}(\widetilde{\Ric}_{\rm sm}+C\pi^*\omega_{\rm can})\wedge(\pi^*\omega_{\rm can})^{n-1},
\end{split}\]
and so, since $\ti{h}$ is continuous, it suffices to show that $(\widetilde{\Ric}_{\rm sm}+C\pi^*\omega_{\rm can})\wedge(\pi^*\omega_{\rm can})^{n-1}$ puts no mass on $E$. Since $\omega_{\rm can}$ has continuous potentials and  $\widetilde{\Ric}_{\rm sm}$ has local potentials with at worst log-log singularities (by Lemma \ref{loglog}), this is then a well-known fact: let $\theta$ be a smooth form on $\ti{N}$ cohomologous to $\widetilde{\Ric}_{\rm sm}+C\pi^*\omega_{\rm can}$, and write $\widetilde{\Ric}_{\rm sm}+C\pi^*\omega_{\rm can}=\theta+\ddbar u\geq 0$ where by Lemma \ref{loglog} the function $u$ satisfies the bounds in \eqref{goalie}. As a consequence of Demailly's regularization \cite[Corollary 6.4]{Dem92}, the cohomology class $[\theta]$ is thus nef, and so for any $\ve>0$ we can find a smooth function $\vp_\ve$ such that $\theta+\ve\omega_{\ti{N}}+\ddbar\vp_\ve$ is a K\"ahler metric on $\ti{N}$. Since $\pi^*\omega_{\rm can}$ has continuous potentials, it follows from Bedford-Taylor \cite{BT} that $\omega_{\ti{N}}\wedge(\pi^*\omega_{\rm can})^{n-1}$ puts no mass on $E$. Thus, to prove our claim it suffices to show that $(\theta+\ve\omega_{\ti{N}}+\ddbar u)\wedge (\pi^*\omega_{\rm can})^{n-1}$ puts no mass on $E$, and this follows e.g. from \cite[Theorem 1.3]{GZ} since $u-\vp_\ve$ belongs to $\mathcal{E}^1(\ti{N},\theta+\ve\omega_{\ti{N}}+\ddbar\vp_\ve)$ as it has at worst log-log singularities, e.g. by \cite[Proposition 2.3]{Gue}.
\end{proof}

The following Proposition uses similar ideas as in \cite[Prop. 5.1]{LS}:

\begin{prop}\label{estimatiomanifesta}
For any continuous nonnegative function $h$ on $N$,
\begin{equation}
\int_\Sigma h(2\pi-\theta(x))d\mathcal{H}^{2n-2} \leq C_n
\sum_i \nu(\Ric,D_i)\int_{D_i}h\omega_{\rm can}^{n-1}.
\end{equation}

\end{prop}

\begin{proof}
Given a small $\ve>0$, we consider the closed subset $\mathcal{S}_\ve\subset \mathcal{S}$ defined in \eqref{se} and let $\Sigma_\ve=\mathcal{S}_\ve\backslash \mathcal{S}^{2n-3}$ (and again remove a further subset of vanishing $\mathcal{H}^{2n-2}$ so that tangent cones at all points of $\Sigma_\ve$ are unique).
Given also a small $\delta>0$,
it suffices to prove
\[
\int_{\Sigma_\ve} h(2\pi-\theta(x))d\mathcal{H}^{2n-2} \leq C_n\sum_i \nu(\Ric,D_i)\int_{D_i}h\omega_{\rm can}^{n-1}+C\delta,
\]
where $C$ does not depend on $\ve,\delta$ but is allowed to depend on $h$, as taking the limit $\delta\to 0$ and $\ve\to 0$ gives the claim.
Given an arbitrarily small $r$ depending on $\delta$, since $\mathcal{S}_\ve\setminus \Sigma_\ve$ has Hausdorff codimension at least 3, we can take a cover with
\begin{equation}\label{complementmeasure}
\mathcal{S}_\ve \setminus \Sigma_\ve\subset \bigcup_i B(y_i, r_i'), \quad \sum_i r_i'^{2n-2}\ll \delta, \quad r_i'<r,
\end{equation}
where here and for the rest of this section, $B(x,r)$ denotes the $d_Z$-geodesic ball centered at $x$ with radius $r$, while $B_j(x,r)$ will denote the $\omega_j$-geodesic ball, and $U_{j,r}$ will be the $r$-neighborhood of $D$ with respect to $\omega_j$.
Since $\mathcal{S}_\ve$ is compact, so is the closed subset $K=\mathcal{S}_\ve\setminus \cup_i B(y_i, r_i')$.

For every $x\in K$, by semicontinuity  we can find a small ball $B(x,r_x)$ with $r_x\ll r$ such that
\[
2\pi-\theta(y)\leq (2\pi- \theta(x))(1-\delta ), \quad \textrm{for all } y\in B(x, 10r_x).
\]
Choosing $r_x$ sufficiently small, we can make the rescaled ball $r_x^{-1} B(x,r_x)$ arbitrarily close to the tangent cone at $x$ in the Gromov-Hausdorff sense. Using \cite[Prop. 3.3]{LS}, for $j$ sufficiently large depending on $x$,
\[
\int_{B_j(x,\eta r_x)} R(\omega_j) \frac{\omega_j^n}{n!} \geq \omega_{2n-2} (2\pi- \theta(x)) (1-\delta) (\eta r_x)^{2n-2},\quad \textrm{for all } \frac{1}{10}<\eta<10.
\]
whence
\[
\int_{B_j(x,\eta r_x)} R(\omega_j) \frac{\omega_j^n}{n!} \geq  \omega_{2n-2}(1-2\delta)(\eta r_x)^{2n-2} \sup_{B(x,10r_x)  }(2\pi-\theta).
\]
By compactness, we can cover $K$ with finitely many such balls $B(x_i,r_i)$ with $r_i=r_{x_i}$, so that the inequalities hold for $j$ large enough independent of $x_i\in K$. Taking a Vitali subcover, we may further assume that $B(x_i,r_i/3)$ are mutually disjoint, so for large enough $j$ we have $B_j(x_i,r_i/4)$ mutually disjoint. Using also the continuity of $h$, for $r$ sufficiently small depending on $\delta$ (and on $h$), and $j$ large enough,
\[
\begin{split}
&\sum_i \omega_{2n-2}(1-3\delta)\left(\frac{r_i}{4}\right)^{2n-2} \sup_{B(x_i,r_i)  }h(2\pi-\theta)  \\
&\leq \sum_i \int_{B_j(x_i,r_i/4)} h (R(\omega_j)+nC) \frac{ \omega_j^n}{n!}
\\
& \leq \int_{ U_{j,r}} h(\Ric(\omega_j) +C\omega_j) \wedge \frac{\omega_j^{n-1} }{(n-1)!}
\\
& \leq \int_{ U_r }  h (\Ric+C\omega_{\rm can}) \wedge \frac{\omega_{\rm can}^{n-1}}{(n-1)!} +\delta.
\end{split}
\]
using \eqref{lollo}.
Combining this with (\ref{complementmeasure}), and take the limit $r\to 0$ using Lemma \ref{cuz} gives
\[\begin{split}
4^{2-2n} \int_{\Sigma_\ve} h(2\pi-\theta(x))d\mathcal{H}^{2n-2} &\leq \lim_{r\to 0}\int_{ U_r } h \Ric \wedge \frac{\omega_{\rm can}^{n-1}}{(n-1)!} +C\delta\\
&=\sum_i \nu(\Ric,D_i)\int_{D_i}h\frac{\omega_{\rm can}^{n-1}}{(n-1)!}+C\delta,
\end{split}\]
as required.\end{proof}

Recall that the singular set satisfies
$$\mathcal{S}=\{x\in N\ |\ \nu(\Ric,x)>0\}\subset D,$$
is an at most countable union of closed analytic subvarieties of $N$. Write $\mathcal{S}=\bigcup_{i'} D_{i'} \cup \mathcal{S}_{\geq 2}$ where $\bigcup_{i'} D_{i'}$ is the (finite) union of divisorial components of $\mathcal{S}$ (which are necessarily also divisorial components of $D$, so equal to a subset of the $D_i$'s, which we have indexed by $i'$ for clarity) and $\mathcal{S}_{\geq 2}$ is an at most countable union of closed irreducible analytic subvarieties of $N$ of complex codimension at least $2$.

\begin{lem}\label{vanish}
If $W\subset \mathcal{S}$ is an at most countable union of closed irreducible analytic subvarieties of $N$ of complex codimension at least $2$, then
$$\mathcal{H}^{2n-2}(W)=0,$$
where as usual $\mathcal{H}$ is the Hausdorff measure of $d_Z$.
\end{lem}
\begin{proof}
It suffices to show that $\mathcal{H}^{2n-2}(V)=0$ for any irreducible component $V$ of $W$, and we can also implicitly remove $\mathcal{S}^{2n-3}$ since it has vanishing Hausdorff measure, so that we can pretend that $V\subset\Sigma$.
Let $h_\ve\geq 0$ be a family of smooth cutoff functions, with $h_\ve$ supported in $B_{g_N}(V,2\ve)$ and $h_\ve\equiv 1$ on $B_{g_N}(V,\ve)$, and applying Proposition \ref{estimatiomanifesta} gives
\[\begin{split}
\int_V (2\pi-\theta(x))d\mathcal{H}^{2n-2}&\leq \int_\Sigma h_\ve(2\pi-\theta(x))d\mathcal{H}^{2n-2} \leq C_n
\sum_i \nu(\Ric,D_i)\int_{D_i}h_\ve\omega_{\rm can}^{n-1}\\
&\leq C_n \sum_i \nu(\Ric,D_i)\int_{B_{g_N}(V,2\ve)}\omega_{\rm can}^{n-1}\to 0,
\end{split}\]
as $\ve\to 0$. This shows that
$$\int_V (2\pi-\theta(x))d\mathcal{H}^{2n-2}=0,$$
but since $\theta(x)<2\pi$ for all $x\in \Sigma$, this gives $\mathcal{H}^{2n-2}(V)=0$.
\end{proof}

\begin{prop}\label{herrlelong}
Let $x\in \Sigma$ be a point with tangent cone $\mathbb{R}^{2n-2}\times \mathbb{C}_{\theta(x)}$. Then the Lelong number of $\Ric$ at $x$ is at most  $2\pi-\theta(x)$.
\end{prop}
We believe that the Lelong number of $\Ric$ at $x$ is actually equal to $2\pi-\theta(x)$, but this does not seem to follow from our arguments below.

\begin{proof}
We will write $\theta=\theta(x)$.
By assumption, the rescaled balls $r^{-1}B(x, r^{1/4})$ converge in the pointed Gromov-Hausdorff sense to the tangent cone $\mathbb{C}^{n-1}_{z_1,\ldots z_{n-1}}\times (\mathbb{C}_\theta)_{z_{n}}$ as $r\to 0$. The metric on the tangent cone is written as
\[
g_\infty=\sum_{i=1}^{n-1}|dz_i|^2 + |z_n|^{-2(1-\theta/2\pi)} |dz_n|^2.
\]
Using \cite[Prop 3.2]{LS}, we can find holomorphic coordinates $w_1, \dots , w_n$ (depending on $r$) on $r^{-1}B(x, r^{1/3})$ converging to $z_1,\dots, z_n$ as $r\to 0$. We can regard $w_i$ also as holomorphic coordinates for the smooth approximating metrics $\omega_j$, because as we know the Gromov-Hausdorff limit $Z$ is homeomorphic to $N$.

 Given any small $\delta>0$, our goal is to show the Lelong number of $\Ric$ at $x\in r^{-1}B(x, r^{1/2})$ is at most $2\pi-\theta+ \delta$.
 By the monotonicity of the Lelong number,  it suffices that for $r\ll 1$,
\[
\frac{1}{\omega_{2n-2}(n-1)!}\int_{ \sum |w_i|^2<1} (\Ric+C\omega_{\rm can}) \wedge \omega_{\C^n}^{n-1} \leq 2\pi-\theta+ \delta,
\]
where $\omega_{\C^n}= \frac{\sqrt{-1}}{2}\sum_{i=1}^n dw_i\wedge d\bar{w}_i.$
Since $\Ric(\omega_j)\to \Ric,\omega_j\to\omega_{\rm can}$ weakly, and $\Ric(\omega_j)+ C\omega_j\geq 0$, this reduces to showing for $j\gg 1$ depending on $\delta, r$,
\begin{equation}\label{g1}
\frac{1}{\omega_{2n-2} (n-1)!}\int_{ \sum |w_i|^2<1} (\Ric(\omega_j)+C\omega_j) \wedge \omega_{\C^n}^{n-1} \leq 2\pi-\theta+ \delta.
\end{equation}
As in \cite{LS}, we use the Cheng-Yau gradient estimate \cite{CY} for the holomorphic functions $z_i, i=1,\cdots,n,$
which on $\{ \sum |w_i|^2<1\}$ gives
\begin{equation}\label{cyg}
\omega_{\C^n}\leq Cr^{-2}\omega_j.
\end{equation}
Using \eqref{cyg} together with the Bishop volume comparison inequality $\mathrm{Vol}_{\omega_j}B_j(x,10r)\leq Cr^{2n}$ (for all $0<r\leq 1$), we can bound
\[
\int_{ \sum |w_i|^2<1} \omega_j\wedge \omega_{\C^n}^{n-1} \leq C\int_{B_j(x,10r)} r^{2-2n} \omega_j^n \leq Cr^2,
\]
Thus, to establish \eqref{g1}, it suffices to show that as $r\to 0$ and $j\to \infty$ fast enough (depending on $r$),
\begin{equation}\label{Lelongnumberlimit}
\frac{1}{\omega_{2n-2}(n-1)!}\lim  \int_{ \sum |w_i|^2<1 } \Ric(\omega_j)\wedge \omega_{\C^n}^{n-1}= 2\pi- \theta.
\end{equation}

We know $r^{-2}g_j$ converge in the pointed Gromov-Hausdorff sense to $\mathbb{C}^{n-1}\times \mathbb{C}_\theta$, and the coordinates $w_i$ converge to $z_i$.
Write $s_j= dw_1\wedge \cdots \wedge dw_n$ and $s= dz_1\wedge\cdots\wedge dz_n$, so
\[
\Ric(\omega_j)= \sqrt{-1}\partial \bar{\partial} \log |s_j|_{g_j}^2,\quad  \Ric(g_\infty)= \sqrt{-1}\partial \bar{\partial} \log |s|_{g_\infty}^2.
\]
From the proof of \cite[Prop 3.3]{LS},
\[
\lim \int_{ \sum |w_i|^2<10 } | \log |s_j|_{g_j}- \log |s|_{g_\infty} | (r^{-2}\omega_j)^n =0.
\]
Using again \eqref{cyg},
\[
\lim \int_{ \sum |w_i|^2<10 } | \log |s_j|_{g_j}- \log |s|_{g_\infty} | \omega_{\C^n}^n =0.
\]
Given any smooth cutoff function $h(z_1,\dots, z_n)$, we can regard it as a function of $w_1,\dots, w_n$. Thus
\[
\begin{split}
&\frac{1}{\omega_{2n-2}(n-1)!}\lim  \int_{ \sum |w_i|^2<1 } h \Ric(\omega_j)\wedge \omega_{\C^n}^{n-1}\\
=
&
\frac{1}{\omega_{2n-2}(n-1)!} \lim \int_{ \sum |w_i|^2<1 } \log |s_j|^2_{g_j} \sqrt{-1}\partial \bar{\partial} h \wedge \omega_{\C^n}^{n-1}
\\
=&
\frac{1}{\omega_{2n-2}(n-1)!} \int_{ \sum |w_i|^2<1 } \log |s|^2_{g_\infty} \sqrt{-1}\partial \bar{\partial} h \wedge \omega_{\C^n}^{n-1}
\\
=& \frac{1}{\omega_{2n-2}(n-1)!} \int_{ \sum |w_i|^2<1 }h \sqrt{-1}\partial \bar{\partial} \log |s|^2_{g_\infty}   \wedge \omega_{\C^n}^{n-1}
\\
=& \frac{2\pi-\theta}{\omega_{2n-2}(n-1)!} \int_{ \sum |w_i|^2<1, w_n=0 }h    \omega_{\C^n}^{n-1}
\end{split}
\]
We let $h$ approach the characteristic function on $\{ \sum |w_i|^2<1  \}$ to obtain (\ref{Lelongnumberlimit}) as required.	\end{proof}

\begin{proof}[Proof of Theorem \ref{cy_smoothbase}]
For each fixed $i'$, let $h_\ve\geq 0$ be a family of smooth cutoff functions supported in $B_{g_N}(D_{i'},2\ve)$ and $h_\ve\equiv 1$ on $B_{g_N}(D_{i'},\ve)$, and applying Proposition \ref{estimatiomanifesta} gives
\begin{equation}\label{e1a}
\begin{split}
\int_{D_{i'}} (2\pi-\theta(x))d\mathcal{H}^{2n-2}&\leq \int_\Sigma h_\ve  (2\pi-\theta(x))d\mathcal{H}^{2n-2}\leq
 C_n
\sum_i \nu(\Ric,D_i)\int_{D_i}h_\ve\omega_{\rm can}^{n-1}\\
&\leq  C_n
\sum_i \nu(\Ric,D_i)\int_{B_{g_N}(D_{i'},2\ve)\cap D_i}\omega_{\rm can}^{n-1},
\end{split}
\end{equation}
and since the RHS converges to $C_n\nu(\Ric,D_{i'})\int_{D_{i'}}\omega_{\rm can}^{n-1}$ as $\ve\to 0$, this gives
\begin{equation}\label{e1}
\int_{D_{i'}} (2\pi-\theta(x))d\mathcal{H}^{2n-2}\leq C_n\nu(\Ric,D_{i'})\int_{D_{i'}}\omega_{\rm can}^{n-1},
\end{equation}
for all $i'$.
But recall that
$$\Sigma=\left(\bigcup_{i'} D_{i'} \cup \mathcal{S}_{\geq 2}\right)\backslash \mathcal{S}^{2n-3}.$$
Let $D^\circ_{i'}$ be points $x$ of the irreducible component $D_{i'}$ where $\nu(\Ric,x)=\nu(\Ric,D_{i'})$, so by Siu \cite{Siu} we know that $D_{i'}\backslash D^\circ_{i'}$  is an at most countable union of closed irreducible analytic subvarieties of $N$ of complex codimension at least $2$. Lemma \ref{vanish} shows that $\mathcal{H}^{2n-2}(\mathcal{S}_{\geq 2})=\mathcal{H}^{2n-2}(D_{i'}\backslash D^\circ_{i'})=0$. Thus
\begin{equation}\label{e2}
\mathcal{H}^{2n-2}(\Sigma)\leq\sum_{i'}\mathcal{H}^{2n-2}(D^\circ_{i'}),
\end{equation}
and
\begin{equation}\label{e2b}
\int_{D_{i'}}(2\pi-\theta(x))d\mathcal{H}^{2n-2}=\int_{D_{i'}^\circ}(2\pi-\theta(x))d\mathcal{H}^{2n-2}.
\end{equation}
On the other hand, Proposition \ref{herrlelong} gives for each $i'$
\begin{equation}\label{e3}
\int_{D^\circ_{i'}}(2\pi-\theta(x))d\mathcal{H}^{2n-2}\geq \nu(\Ric,D_{i'})\mathcal{H}^{2n-2}(D^\circ_{i'}),
\end{equation}
and combining \eqref{e1}, \eqref{e2b} and \eqref{e3} we deduce that for each $i'$
\[
\mathcal{H}^{2n-2}(D^\circ_{i'})\leq C_n\int_{D_{i'}}\omega_{\rm can}^{n-1}=C_n\int_{D_{i'}}\omega_{N}^{n-1},
\]
and with \eqref{e2} we finally deduce that
\[
\mathcal{H}^{2n-2}(\Sigma) \leq C_n\sum_{i'} \int_{D_{i'}}\omega_N^{n-1}\leq C_n\int_{D} \omega_N^{n-1},
\]
where $D$ is regarded as a reduced divisor.
\end{proof}

\section{Collapsing and the canonical bundle formula}\label{sectbir}
\subsection{Volume form asymptotics}
We now discuss the estimate \eqref{desired}. We again work in the unified setting \eqref{ma0},
where $\lambda=0$ in the Calabi-Yau setup and $\lambda=1$ in the K\"ahler-Ricci flow setup.

It was shown in \cite{GTZ3} that estimate \eqref{desired} holds if $N$ is smooth and $D$ is a simple normal crossings divisor.
We thus assume that this is not the case, and let $\pi:\ti{N}\to N$ be a sequence of blowups with smooth centers such that $\ti{N}$ is smooth and $E=\pi^{-1}(D)$ is a divisor with simple normal crossings. Following the construction in the proof of \cite[Theorem 2.3]{GTZ3}, we consider a resolution of singularities $\ti{M}\to M\times_N\ti{N}$ (birational onto the main component of the target space) and obtain the commutative diagram
\[
\xymatrix@C=30pt
{  \ti{M}\ar[dr]_{\ti{f}}\ar[r]\ar@/^1pc/[rr]^{p} & M\times_N \ti{N}
\ar[r]\ar[d]&M\ar[d]^{f}\\
&\ti{N}\ar[r]_{\pi}&N}
\]
where $\ti{M}^m$ is smooth. Since $M$ is also smooth, we can write $K_{\ti{M}}\sim p^*K_M+\ti{D}$ where $\ti{D}$ is an effective $p$-exceptional divisor, which  can be assumed to have simple normal crossings support.
The volume form $\ti{\mathcal{M}}:=p^*\mathcal{M}$ on $\ti{M}$ is smooth and in general has zeros along $\ti{D}$.
If we define $\ti{\vp}=\pi^*\vp\in C^0(\ti{N})\cap C^\infty(\ti{N}\backslash E)$, then on $\ti{N}\backslash E$ we have
$$\pi^*\omega_{\rm can}^n=e^{\lambda\ti{\vp}} \ti{f}_*(\ti{\mathcal{M}})=\pi^*\left(e^{\lambda\vp}f_*(\mathcal{M})\right),$$
and the asymptotic behavior of the volume form $\pi^*\omega_{\rm can}^n$ was obtained in \cite[Theorems 2.3 and 7.1]{GTZ3} using Hodge theory (and in \cite{Ki} with a different method, which also extends to the case when the morphism $f$ is K\"ahler but not projective, see \cite[Rmk 1.6]{Ki}): on $\ti{N}\backslash E$ we have
\begin{equation}\label{asym}
C^{-1}\prod_{j=1}^p |s_j|_{h_j}^{2\beta_j}\omega_{\rm cone}^n\leq \pi^*\omega_{\rm can}^n\leq C\prod_{j=1}^p|s_j|_{h_j}^{2\beta_j}\left(1-\sum_{i=1}^\mu\log|s_i|_{h_i}\right)^d \omega_{\rm cone}^n,
\end{equation}
where $\beta_j\in\mathbb{Q}_{>0}$, and where $\omega_{\rm cone}$ is a K\"ahler metric with conical singularities along the $E_i$'s with cone angles $2\pi\gamma_i$, $0<\gamma_i\leq 1$, which we will take of the form $\omega_{\rm cone}=\omega_{\ti{N}}+\ddbar\eta,$ where
\begin{equation}\label{defneta}
\eta=C^{-1}\sum_i |s_i|^{2\gamma_i}_{h_i},
\end{equation}
for some $C>0$ sufficiently large. In particular we have
$$\frac{C^{-1}}{\prod_i |s_i|^{2(1-\gamma_i)}_{h_i}}\leq\frac{\omega_{\rm cone}^n}{\omega_{\ti{N}}^n}\leq \frac{C}{\prod_i |s_i|^{2(1-\gamma_i)}_{h_i}}.$$
Let us write
$$H=\prod_j |s_j|_{h_j}^{2\beta_j},$$
and define a smooth function $\psi$ on $\ti{N}\backslash E$ by
\begin{equation}\label{psi}
\psi=\frac{\pi^*\omega_{\rm can}^n\prod_i |s_i|^{2(1-\gamma_i)}_{h_i}}{e^{\lambda\ti{\vp}}H\omega_{\ti{N}}^n},
\end{equation}
which depends on the choice of Hermitian metrics $h_i$, and which by \eqref{asym} and the boundedness of $\ti{\vp}$ satisfies
\begin{equation}\label{ss}
C^{-1}\leq \psi\leq C\left(1-\sum_{i=1}^\mu\log|s_i|_{h_i}\right)^d.
\end{equation}

\subsection{The canonical bundle formula}
The exponents $\beta_j,\gamma_i$ in \eqref{asym} can be determined by applying the canonical bundle formula in birational geometry \cite{Am,FL,FM,Ka,Ki,Kol,FL} to the map $\ti{f}$. Following the notation in \cite{Ki}, we define divisors $\ti{R}=-\ti{D}$ on $\ti{M}$ and $\mathfrak{M}=-K_{\ti{N}}$ on $\ti{N}$, so that we have the equality as $\mathbb{Q}$-divisors
$$K_{\ti{M}}+\ti{R}=\ti{f}^*(K_{\ti{N}}+\mathfrak{M}).$$
We also define $\ti{B}=\pi^{-1}(D)\subset\ti{N}$, and note that $\ti{R}+\ti{f}^*\ti{B}$ has snc support, and $\ti{f}(\mathrm{Supp}\ti{R})\subset\ti{B}$ (so in particular $\ti{R}$ is vertical, with the terminology of \cite{Ki}). It then follows that $\ti{f}$ satisfies the conditions in \cite[Definition 4.3]{Ki}, and thus thanks to \cite[(16)]{Ki}, \cite[Theorem 8.3.7]{Kol} there is a well-defined $\mathbb{Q}$-divisor $\ti{B}_{\ti{R}}$ on $\ti{N}$ supported on $\ti{B}$, the {\em boundary part of the canonical bundle formula} for $\ti{f}$, which satisfies
$$\ti{R}+\ti{f}^*(\ti{B}-\ti{B}_{\ti{R}})\leq \mathrm{red}(\ti{f}^*\ti{B}),$$
and is the smallest such divisor. Writing
\begin{equation}\label{coeffs}
\ti{B}_{\ti{R}}=\sum_i a_i \ti{B}_i,
\end{equation}
it follows from \cite[(16)]{Ki} that $a_i\in (-\infty,1)$.

Then $\ti{\mathcal{M}}$ is a volume form on $\ti{M}$ with ``poles along $R$'' in the terminology of \cite{Ki} (i.e. zeros along $\ti{D}=-R$), so \cite[Corollary 1.3]{Ki} applies (beware that there is a typo in \cite[(4)]{Ki}, and the exponents $a_i$ there should be replaced by $-a_i$) and shows that
$$\pi^*\omega_{\rm can}^n=e^{\lambda\ti{\vp}} \ti{f}_*(\ti{\mathcal{M}}),$$
on $\ti{N}\backslash E$ satisfies
$$C^{-1}\prod_i |s_i|^{-2a_i}_{h_i}\psi\leq \frac{\pi^*\omega_{\rm can}^n}{\omega_{\ti{N}}^n}\leq C\prod_i |s_i|^{-2a_i}_{h_i}\psi,$$
(using again the boundedness of $\ti{\vp}$) where $s_i$ is a defining section of $\mathcal{O}(\ti{B}_i)$, the coefficients $a_i$ are given by \eqref{coeffs} and $\psi$ is as in \eqref{ss}. Comparing this with \eqref{asym} shows that the exponents $\beta_j$ in \eqref{asym} are just equal to $-a_i$ for those $a_i<0$, and the exponents $(1-\gamma_i)$ in \eqref{asym} are equal to $a_i$ for those $a_i>0$.

Given thus Hermitian metrics $h_i$ on $\mathcal{O}(\ti{B}_i)$ (which we will choose precisely later), we define $\psi$ as in \eqref{psi} by
$$\psi=\frac{\pi^*\omega_{\rm can}^n\prod_i |s_i|^{2a_i}_{h_i}}{e^{\lambda\ti{\vp}}\omega_{\ti{N}}^n},$$
so that on $\ti{N}\backslash E$
\begin{equation}\label{ddb}
\begin{split}
\ddbar\log(1/\psi)&=\Ric(\pi^*\omega_{\rm can})-\Ric(\omega_{\ti{N}})+\sum_i a_i R_{h_i}+\lambda\pi^*\ddbar\vp\\
&=\pi^*\omega_{\rm WP}-\lambda\pi^*\omega_{\rm can}-\Ric(\omega_{\ti{N}})+\sum_i a_i R_{h_i}+\lambda\pi^*\ddbar\vp\\
&=\pi^*\omega_{\rm WP}-\lambda\pi^*\omega_N-\Ric(\omega_{\ti{N}})+\sum_i a_i R_{h_i}\\
&\geq -\lambda\pi^*\omega_N-\Ric(\omega_{\ti{N}})+\sum_i a_i R_{h_i},
\end{split}
\end{equation}
since $\omega_{\rm WP}\geq 0$ on $N^\circ$. Observe that all terms on the last line of \eqref{ddb} are smooth forms on all of $\ti{N}$, and
the term $\sum_i a_i R_{h_i}$ is cohomologous to $\ti{B}_{\ti{R}}$.

\subsection{Vanishing orders}
In this section we will use repeatedly the notion of a K\"ahler metric $\omega_N$ on a singular (reduced, irreducible) compact complex analytic space $N$, as in \cite{Moi}, see also \cite[Chapter XII.3]{BM}. This has the property that if $\pi:\ti{N}\to N$ is a resolution of singularities then $\pi^*\omega_N$ is a smooth semipositive $(1,1)$-form on $\ti{N}$. Furthermore, the resolution $\ti{N}$ can be chosen to be a K\"ahler manifold and if $\omega_{\ti{N}}$ is any fixed K\"ahler metric on $\ti{N}$ then
$$\frac{\pi^*\omega_N^n}{\omega_{\ti{N}}^n}$$
is a smooth semipositive function on $\ti{N}$ which vanishes precisely along the exceptional locus $\mathrm{Exc}(\pi)$. We may assume without loss that $\mathrm{Exc}(\pi)=\bigcup_k F_k$ is a simple normal crossings divisor, and we can find real numbers $b_k\in\mathbb{R}_{>0}$ such that the ratio
$$\frac{\pi^*\omega_N^n}{\prod_k |s_{F_k}|^{2b_k}_{h_{F_k}}\omega_{\ti{N}}^n}$$
is a smooth strictly positive function on $\ti{N}$ (for any smooth Hermitian metrics $h_{F_k}$ on $\mathcal{O}(F_k)$). By Yau's Theorem \cite{Ya}
we can pick our reference K\"ahler metric $\omega_{\ti{N}}$ such that we have
\begin{equation}\label{grz}
\frac{\pi^*\omega_N^n}{\omega_{\ti{N}}^n}=\prod_k |s_{F_k}|^{2b_k}_{h_{F_k}}.
\end{equation}
Observe that the coefficients $b_k$ are unchanged if we replace $\omega_N$ by another K\"ahler metric on $N$, since the pullbacks of these metrics to $\ti{N}$ are uniformly equivalent: indeed, given two K\"ahler metrics $\omega_N,\omega'_N$ on $N$, given any $x\in N$ we can find an open set $U\ni x$ in $N$ and embeddings $\iota:U\hookrightarrow \mathbb{C}^N, \iota':U\hookrightarrow \mathbb{C}^{N'}$ and smooth strictly psh functions $\vp,\vp'$ defined in some neighborhoods of the images $\iota(U),\iota'(U')$ such that $\omega_N|_U=\iota^*\ddbar\vp,\omega'_N|_U=\iota'^*\ddbar\vp'$. Then \cite[Lemma XI.1.3.2]{BM} shows that, up to shrinking our neighborhoods, we may assume that the embeddings $\iota$ and $\iota'$ are equal, and then it is clear that $\ddbar\vp$ and $\ddbar\vp'$ are locally uniformly equivalent, and pulling back via $\iota$ and $\pi$ shows that $\pi^*\omega_N$ and $\pi^*\omega'_N$ are uniformly equivalent, proving the claim. We can thus define a $\pi$-exceptional $\mathbb{R}$-divisor on $\ti{N}$
$$\mathcal{I}_{\ti{N}/N}=\sum_k b_k F_k,$$
which does not depend on the choice of $\omega_N$. When $N$ is smooth, we have that $\mathcal{I}_{\ti{N}/N}=K_{\ti{N}/N}$, but this equality does not hold in general (say when $N$ is $\mathbb{Q}$-Gorenstein so that $K_N$ is $\mathbb{Q}$-Cartier) since the discrepancies of $\pi$ can be negative while the vanishing orders $b_k$ are always positive).

If $$\hat{N}\overset{\hat{\pi}}{\to}\ti{N}\overset{\pi}{\to}N$$
is a higher model, then fixing a K\"ahler metric $\omega_{\hat{N}}$ on $\hat{N}$
we have
$$\frac{\hat{\pi}^*\pi^*\omega_N^n}{\omega_{\hat{N}}^n}=\hat{\pi}^*\left(\frac{\pi^*\omega_N^n}{\omega_{\ti{N}}^n}\right)\frac{\hat{\pi}^*\omega_{\ti{N}}^n}{\omega_{\hat{N}}^n},$$
and so
\begin{equation}\label{key}
\mathcal{I}_{\hat{N}/N}=\hat{\pi}^*\mathcal{I}_{\ti{N}/N}+K_{\hat{N}/\ti{N}}.
\end{equation}

\subsection{A functorial divisor}
We then define a $\mathbb{Q}$-divisor on $\ti{N}$ by
$$\Xi_{\ti{N}}:=\ti{B}_{\ti{R}}+\mathcal{I}_{\ti{N}/N}.$$

If we are now given a higher model $\hat{\pi}:\hat{N}\to\ti{N}$, and construct $\hat{f}:\hat{M}\to\hat{N}$ as above, then we have (see e.g. \cite[Lemma 4.10]{Ki})
\begin{equation}\label{key2}
\hat{B}_{\hat{R}}=\hat{\pi}^*\ti{B}_{\ti{R}}-K_{\hat{N}/\ti{N}},
\end{equation}
and combining \eqref{key} and \eqref{key2} we obtain the functorial relation
\begin{equation}\label{key3}
\Xi_{\hat{N}}=\hat{B}_{\hat{R}}+\mathcal{I}_{\hat{N}/N}=\hat{\pi}^*\ti{B}_{\ti{R}}-K_{\hat{N}/\ti{N}}+\hat{\pi}^*\mathcal{I}_{\ti{N}/N}+K_{\hat{N}/\ti{N}}=\hat{\pi}^*\Xi_{\ti{N}}.
\end{equation}

\subsection{Collapsing}
Having introduced the divisor $\Xi_{\ti{N}}$, we now come to the proof of Theorem \ref{cy_canonical}, which we restate here:

\begin{thm}
Suppose that there is a resolution $\pi:\ti{N}\to N$ as above such that $\Xi_{\ti{N}}$ is $\pi$-ample. Then the conjectured estimate \eqref{desired} holds on $\ti{N}\backslash E$.
\end{thm}

\begin{proof}
We start the proof by using the method of \cite{GTZ3}.
From \eqref{ss}, $\log(1/\psi)$ is bounded above near $E$, so by Grauert-Remmert \cite{GR} it extends to a global quasi-psh function on $\ti{N}$ which satisfies \eqref{ddb} in the weak sense. Thanks to \eqref{ss}, the extension has vanishing Lelong numbers, so we can approximate it using Demailly's regularization theorem \cite{De} by a decreasing sequence of smooth functions $u_j$ with arbitrarily small loss of positivity, i.e.
\begin{equation}\label{zwei}
\ddbar u_j\geq  -\lambda\pi^*\omega_N-\Ric(\omega_{\ti{N}})+\sum_i a_i R_{h_i}-\frac{1}{j}\omega_{\ti{N}},
\end{equation}
on all of $\ti{N}$.
We use this to obtain a partial regularization of $\pi^*\omega_{\rm can}$, which we denote by $\omega_{j}=\pi^*\omega_N+\frac{1}{j}\omega_{\ti{N}}+\ddbar\vp_{j}$. These are K\"ahler metrics on $\ti{N}\backslash E$ solving
\begin{equation}\label{ma2}
\omega_{j}^n=c_{j} e^{\lambda\vp_j-u_j}\frac{\omega_{\ti{N}}^n}{\prod_i |s_i|^{2a_i}},
\end{equation}
with the normalization $\sup_{\ti{N}}\vp_j=0$ if $\lambda=0$, and where $c_j=1$ for $\lambda=1$, while for $\lambda=0$ the constant $c_j$ is defined by integrating the equation, and satisfies $c_j\to 1$ as $j\to\infty$. This equation is solved via a standard approximation procedure (see e.g. \cite[\S 5]{GTZ3}), and we obtain $\vp_j$ which is smooth on $\ti{N}\backslash E$ and continuous on $\ti{N}$, and as in Section \ref{sect2} we have the properties that
$\omega_j\to \pi^*\omega_{\rm can}$ locally smoothly on $\ti{N}\backslash E$, and $$\sup_{\ti{N}}|\vp_j|\leq C,$$
and $\|\vp_j-\pi^*\vp\|_{L^1(\ti{N},\omega_{\ti{N}}^n)}\to 0.$
Crucially, it is also shown in \cite[Proposition 5.1]{GTZ3} that for each $j$ there is a constant $C_j$ such that on $\ti{N}\backslash E$ we have
\begin{equation}\label{lame}
\tr{\omega_{\rm cone}}{\omega_{j}}\leq C_j,
\end{equation}
so these partial regularizations are not worse than conical (the proof in \cite[Proposition 5.1]{GTZ3} is written with $\lambda=0$, but it extends immediately to the case $\lambda=1$). Also, differentiating \eqref{ma2} and using \eqref{zwei}, we see that on $\ti{N}\backslash E$ we have
\begin{equation}\label{semipos}\begin{split}
\Ric(\omega_j)&=-\lambda\ddbar\vp_j+\ddbar u_j+\Ric(\omega_{\ti{N}})-\sum_i a_iR_{h_i}\\
&\geq-\lambda\omega_j +\lambda\pi^*\omega_N+\frac{\lambda}{j}\omega_{\ti{N}} -\lambda\pi^*\omega_N-\frac{1}{j}\omega_{\ti{N}}\geq -\lambda\omega_j-\frac{C}{j}\omega_{\rm cone}.
\end{split}\end{equation}
Our goal is then to show there are $C,A>0$ such that on $\ti{N}\backslash E$ we have
\begin{equation}\label{est3}
\tr{\omega_{\rm cone}}{\omega_{j}}\leq Ce^{-Au_j},
\end{equation}
holds for all $j$ sufficiently large, since then passing to the limit in $j$ this gives
$$\tr{\omega_{\rm cone}}{\pi^*\omega_{\rm can}}\leq C\psi^A,$$
on $\ti{N}\backslash E$, which is our desired estimate \eqref{desired}.

First, following \cite{GP} we define $\Psi=C\sum_i |s_i|_{h_i}^{2\rho},$ for some small $\rho>0$ and large $C>0$, which can be chosen so that on $\ti{N}\backslash E$ the curvature of $\omega_{\rm cone}$ satisfies
\begin{equation}\label{curv}
{\rm Rm}(\omega_{\rm cone})\geq -(C\omega_{\rm cone}+\ddbar\Psi)\otimes \mathrm{Id},
\end{equation}
see \cite[(4.3)]{GP}.

To prove \eqref{est3} we apply the maximum principle to
$$Q=\log\tr{\omega_{\rm cone}}{\omega_j}+n\Psi+Au_j-A^2(\vp_j-\eta/j)+Ab\eta+\ve \log |s_E|^2,$$
where $A$ is large (to be determined), $b>0$ is small and $0<\ve\leq\frac{1}{j}$, $\eta$ was defined in \eqref{defneta}, and $j$ will be taken larger than $A$ (once the value of $A$ is fixed). The terms
$n\Psi+Au_j-A^2(\vp_j-\eta/j)+Ab\eta$ are all bounded on $\ti{N}$ (with bounds independent of $j$ except for $u_j$), while the term $\log\tr{\omega_{\rm cone}}{\omega_j}$ is bounded above on $\ti{N}\backslash E$ (depending on $j$) by \eqref{lame}. Since the term $\ve \log |s_E|^2$ goes to $-\infty$ on $E$, the quantity $Q$ achieves a global maximum on $\ti{N}\backslash E$. All the following computations are at an arbitrary point of $\ti{N}\backslash E$.

First, from \cite[(5.17)]{GTZ3} we have
$$\Delta_{\omega_j}(\log\tr{\omega_{\rm cone}}{\omega_j}+n\Psi)\geq -C\tr{\omega_j}{\omega_{\rm cone}}-\frac{\tr{\omega_{\rm cone}}{\Ric(\omega_j)}}{\tr{\omega_{\rm cone}}{\omega_j}},$$
while differentiating \eqref{ma2} gives
\[\begin{split}
\Delta_{\omega_j}u_j&=\lambda\Delta_{\omega_j}\vp_j+\tr{\omega_j}{\Ric(\omega_j)}-\tr{\omega_j}{\Ric(\omega_{\ti{N}})}+\tr{\omega_j}\left({\sum_i a_i R_{h_i}}\right)\\
&=\lambda n-\lambda\tr{\omega_j}{\pi^*\omega_N}-\frac{\lambda}{j}\tr{\omega_j}{\omega_{\ti{N}}} +\tr{\omega_j}{\Ric(\omega_j)}-\tr{\omega_j}{\Ric(\omega_{\ti{N}})}+\tr{\omega_j}\left({\sum_i a_i R_{h_i}}\right)\\
&\geq\lambda n-\lambda\tr{\omega_j}{\pi^*\omega_N}-\frac{C}{j}\tr{\omega_j}{\omega_{\rm cone}} +\tr{\omega_j}{\Ric(\omega_j)}-\tr{\omega_j}{\Ric(\omega_{\ti{N}})}+\tr{\omega_j}\left({\sum_i a_i R_{h_i}}\right),
\end{split}\]
and as in \cite[(5.20)]{GTZ3} we observe that
\[\begin{split}
&-\frac{\tr{\omega_{\rm cone}}{\Ric(\omega_j)}}{\tr{\omega_{\rm cone}}{\omega_j}}+\tr{\omega_j}{\Ric(\omega_j)}\\
&=-\frac{\tr{\omega_{\rm cone}}{(\Ric(\omega_j)+\frac{C}{j}\omega_{\rm cone}+\lambda\omega_j)}}{\tr{\omega_{\rm cone}}{\omega_j}}+\tr{\omega_j}{\left(\Ric(\omega_j)+\frac{C}{j}\omega_{\rm cone}+\lambda\omega_j\right)}\\
&+\frac{\tr{\omega_{\rm cone}}{(\frac{C}{j}\omega_{\rm cone}+\lambda\omega_j)}}{\tr{\omega_{\rm cone}}{\omega_j}}
-\frac{C}{j}\tr{\omega_j}{\omega_{\rm cone}}-\lambda n\\
&\geq -\frac{C}{j}\tr{\omega_j}{\omega_{\rm cone}}-\lambda n,
\end{split}\]
using that $\Ric(\omega_j)+\frac{C}{j}\omega_{\rm cone}+\lambda\omega_j\geq 0$ by \eqref{semipos}, so the quantity in the second line is nonnegative. Therefore, using again \eqref{semipos},
\begin{equation}\label{c2}\begin{split}
&\Delta_{\omega_j}(\log\tr{\omega_{\rm cone}}{\omega_j}+n\Psi+Au_j)\\
&\geq-\left(C+\frac{C}{j}\right)\tr{\omega_j}{\omega_{\rm cone}}-\lambda n
+A\lambda n-A\lambda\tr{\omega_j}{\pi^*\omega_N}-\frac{CA}{j}\tr{\omega_j}{\omega_{\rm cone}}\\
&+(A-1)\tr{\omega_j}{\Ric(\omega_j)}
-A\tr{\omega_j}{\Ric(\omega_{\ti{N}})}+A\tr{\omega_j}{\left(\sum_i a_i R_{h_i}\right)}\\
&\geq -\left(C+\frac{CA}{j}\right)\tr{\omega_j}{\omega_{\rm cone}}-A\lambda\tr{\omega_j}{\pi^*\omega_N}
-A\tr{\omega_j}{\Ric(\omega_{\ti{N}})}+A\tr{\omega_j}{\left(\sum_i a_i R_{h_i}\right)},\end{split}
\end{equation}
and
taking $\ddbar\log$ of \eqref{grz} on $\ti{N}\backslash E$ gives
\begin{equation}\label{deriv}
\Ric(\omega_{\ti{N}})=\pi^*\Ric(\omega_N)-\sum_k b_k R_{F_k}.
\end{equation}
To bound the term $\pi^*\Ric(\omega_N)$ we use the following lemma:

\begin{lem}
There is a constant $C$ such that on $\ti{N}$ we have
\begin{equation}\label{ric}\Ric(\omega_{\ti{N}})\leq C\pi^*\omega_N-\sum_k b_k R_{F_k}.
\end{equation}
\end{lem}
\begin{proof}
In \eqref{deriv} the terms $\Ric(\omega_{\ti{N}}), \sum_k b_k R_{F_k}$ and $\pi^*\omega_N$ are smooth on all of $\ti{N}$, so it suffices to show that on $N^\circ$ we have
$$\Ric(\omega_N)\leq C\omega_N.$$
This is of course clear if $N$ is smooth, while for singular $N$ recall that by definition we can cover $N$ by open subsets $U_i$ with embeddings $U_i\hookrightarrow B\subset\mathbb{C}^N$ as analytic subsets of the unit ball in Euclidean space, and on each $U_i$ the metric $\omega_N$ equals the restriction of some K\"ahler metric on $B$. Since bisectional curvature decreases in submanifolds, on $U_i\cap N^{\rm reg}\supset U_i\cap N^\circ$ we have that the bisectional curvature of $\omega_N$ is bounded above, and hence so is its Ricci curvature.
\end{proof}

Inserting \eqref{deriv} and \eqref{ric} in \eqref{c2} then gives
\begin{equation}\label{c3}\begin{split}
\Delta_{\omega_j}(\log\tr{\omega_{\rm cone}}{\omega_j}+n\Psi+Au_j)&\geq-\left(C+\frac{CA}{j}\right)\tr{\omega_j}{\omega_{\rm cone}}-CA\tr{\omega_j}{\pi^*\omega_N}\\
&-A\lambda n+A\tr{\omega_j}{\left(\sum_k b_kR_{F_k}+\sum_i a_i R_{h_i}\right)},\end{split}
\end{equation}
where the term $\sum_k b_kR_{F_k}+\sum_i a_i R_{h_i}$ is the curvature of a Hermitian metric on our divisor $\Xi_{\ti{N}}$.

By assumption $\Xi_{\ti{N}}$ is $\pi$-ample, and so we can choose the metrics $h_{F_k},h_i$ so that
$$\hat{\omega}_{\ti{N}}:=A_0\pi^*\omega_N+\sum_k b_kR_{F_k}+\sum_i a_i R_{h_i}$$
is a K\"ahler metric on $\ti{N}$ for some (in fact all) $A_0$ sufficiently large. We also choose $A$ in the quantity $Q$ so that $A\geq 2A_0$.

Using \eqref{c3} we can then compute
\[\begin{split}
\Delta_{\omega_j}Q&\geq -\left(C+\frac{CA}{j}\right)\tr{\omega_j}{\omega_{\rm cone}}-CA\tr{\omega_j}{\pi^*\omega_N}+A\tr{\omega_j}{\left(\sum_k b_kR_{F_k}+\sum_i a_i R_{h_i}\right)}\\
&-A\lambda n+A^2\tr{\omega_j}{\left(\pi^*\omega_N+\frac{1}{j}\omega_{\rm cone}\right)}+Ab\tr{\omega_j}{\ddbar\eta}-A^2n-\ve\tr{\omega_j}{R_E}\\
&\geq -\left(C+\frac{CA}{j}\right)\tr{\omega_j}{\omega_{\rm cone}}+A\tr{\omega_j}{\left(\sum_k b_kR_{F_k}+\sum_i a_i R_{h_i}\right)}-A\lambda n\\
&+\frac{A^2}{2}\tr{\omega_j}{\pi^*\omega_N}+\frac{A^2}{j}\tr{\omega_j}{\omega_{\rm cone}}+Ab\tr{\omega_j}{\ddbar\eta}-A^2n-\frac{C}{j}\tr{\omega_j}{\omega_{\rm cone}}\\
&\geq  -C\tr{\omega_j}{\omega_{\rm cone}}+A\tr{\omega_j}{\left(A_0\pi^*\omega_N+\sum_k b_kR_{F_k}+\sum_i a_i R_{h_i}+b\ddbar\eta\right)}-A\lambda n-A^2n,
\end{split}\]
assuming without loss that $A$ is large so that $\frac{A^2}{2}\geq CA$ and also that $j\geq A$. Then we choose $b>0$ small so that
$\hat{\omega}_{\ti{N}}+b\ddbar\eta=\hat{\omega}_{\rm cone}$ is a conical K\"ahler metric with $\hat{\omega}_{\rm cone}\geq c\omega_{\rm cone}$ for some $c>0$, so that
$$A\tr{\omega_j}{\left(A_0\pi^*\omega_N+\sum_k b_kR_{F_k}+\sum_i a_i R_{h_i}+b\ddbar\eta\right)}\geq Ac\, \tr{\omega_j}{\omega_{\rm cone}},$$
 and finally we can choose $A$ sufficiently large so that
$$ Ac\, \tr{\omega_j}{\omega_{\rm cone}}\geq (C+1) \tr{\omega_j}{\omega_{\rm cone}},$$
and so we obtain
$$\Delta_{\omega_j}Q\geq \tr{\omega_j}{\omega_{\rm cone}}-C.$$
Therefore at a maximum of $Q$ (which is not on $E$) we have
$$\tr{\omega_j}{\omega_{\rm cone}}\leq C,$$
and so also
$$\tr{\omega_{\rm cone}}{\omega_j}\leq C H e^{-u_j},$$
hence
$$\log\tr{\omega_{\rm cone}}{\omega_j}+Au_j\leq C\log H+(A-1)u_j\leq C,$$
and so also $Q\leq C$, which must hold everywhere on $\ti{N}\backslash E$. The constants do not depend on $\ve$, so we can let $\ve\to 0$ and this gives
$$\tr{\omega_{\rm cone}}{\omega_j}\leq Ce^{-Au_j},$$
which is \eqref{est3}.
\end{proof}

\section{Collapsing of the K\"ahler-Ricci flow}\label{sectkrf}
In this section we give the proof of Theorem \ref{krf_gh}. The setup was described in detail in section \ref{setupkrf} in the Introduction, and we will not repeat it here.

\subsection{Review of some recent results}
We first collect some recent results from the literature that will be used in the course of our proof.

First, by \cite[Theorem 1.2]{TWY}  we have that
\begin{equation}\label{twyconv}
\omega(t)\to f^*\omega_{\rm can},
\end{equation}
in $C^0_{\rm loc}(M^\circ)$, while \cite{ST2} shows that the scalar curvature of $\omega(t)$ is uniformly bounded, i.e.
\begin{equation}\label{stscal}
\sup_M|R(\omega(t))|\leq C,
\end{equation}
for all $t\geq 0$, and also that the volume form of $\omega(t)$ satisfies
\begin{equation}\label{volscal}
C^{-1}e^{-(m-n)t}\omega_M^{m}\leq\omega(t)^{m}\leq Ce^{-(m-n)t}\omega_M^{m},
\end{equation}
on $M\times [0,\infty),$ as well as the ``parabolic Schwarz Lemma'' estimate \cite[Proposition 2.2]{ST2} (and also \cite[(3.5)]{TZ3} for the case when $N$ is singular)
\begin{equation}\label{schwarz}
\omega(t)\geq C^{-1}f^*\omega_N,
\end{equation}
on $M\times [0,\infty).$

Next, using the results in \cite{Bam}, in \cite[Theorem 1.1]{JS} it was very recently proved that
\begin{equation}\label{diambd}
\mathrm{diam}(M,g(t))\leq C,
\end{equation}
uniformly for all $t\geq 0$.
Also, in \cite[Corollary 1.1]{JS} it is shown that there is a uniform $C$ such that for all $x\in M, 0<r<\mathrm{diam}(M,g(t)),t\geq 0$ we have
\begin{equation}\label{volbd}
C^{-1}e^{-(m-n)t}\leq \frac{\mathrm{Vol}_{g(t)}B^{g(t)}(x,r)}{r^{2n}}\leq Ce^{-(m-n)t}.
\end{equation}

We also have information about the collapsed limit space $(N^\circ, \omega_{\rm can})$. Thanks to our assumptions that $N$ is smooth and $D^{(1)}$ has simple normal crossings, we can apply \cite[Theorem 1.4]{GTZ3} which gives us that
\begin{equation}\label{quasiisom}
C^{-1}\omega_{\rm cone}\leq\omega_{\rm can}\leq C\left(1-\sum_{i=1}^\mu\log|s_i|_{h_i}\right)^A \omega_{\rm cone},
\end{equation}
holds on $N\backslash D^{(1)}$, where $\omega_{\rm cone}$ is a K\"ahler metric with conical singularities along $D^{(1)}=\bigcup_i D_i$.
Also, in \cite[Prop.3.1]{GTZ3} it is shown that $\omega_{\rm can}$ extends to a smooth K\"ahler metric across $D^{(2)}$, so without loss we may assume that $D=D^{(1)}$ is a simple normal crossings divisor.

If we denote by $d_{\rm can}$ the associated distance function on $N^\circ$, then it is shown in \cite[Theorem 6.1]{GTZ3} (using \cite[(2.7)]{TZ2}) and also in \cite[Proposition 2.2]{STZ} that $(N^\circ,d_{\rm can})$ has finite diameter, and so its metric completion $(Z,d_Z)$ is a compact metric space which by \cite[Proposition 2.3]{STZ} is homeomorphic to $N$ (here we use that $N$ is smooth).
Also, \cite[Proposition 2.2]{STZ} shows that for every $p,q\in N^\circ$ and $\delta>0$ there is a path $\gamma$ in $N^\circ$ joining $p$ and $q$ with
\begin{equation}\label{almostc}
L_{g_{\rm can}}(\gamma)\leq d_Z(p,q)+\delta.
\end{equation}
We can call this the ``almost convexity'' of $(N^\circ,d_{\rm can})$ inside its metric completion.

There is also a more localized version of this almost convexity. Let us introduce the following notation: for any $\ve>0$ let $U_\ve\subset N$ be the $\ve$-neighborhood of $D$ with respect to the fixed metric $\omega_N$ on $N$, and let $\ti{U}_\ve=f^{-1}(U_\ve)\subset M$.
Then in \cite[Proposition 2.2]{STZ} it is shown that given any $\delta,\ve'>0$ sufficiently small, there is $0<\ve<\ve'$ such that for every $p',q'\in N\backslash U_{\ve'}$ there is a path $\gamma$ in $N\backslash U_{\ve}$ joining $p'$ and $q'$ such that \eqref{almostc} holds.
But thanks to the upper bound in \eqref{quasiisom} it follows that for every $p\in N\backslash U_{\ve}$ there is $p'\in N\backslash U_{\ve'}$ which can be joined to $p$ by a path which is contained in $N\backslash U_{\ve}$ and with $g_{\rm can}$-length at most $\delta$. Concatenating this path, the path $\gamma$, and the analogous path joining $q$ and $q'$ insider $N\backslash U_\ve$ we conclude that given $\delta>0$ there is $\ve>0$ such that for every $p,q\in N\backslash U_{\ve}$ there is a path $\gamma$ in $N\backslash U_{\ve}$ joining them such that \eqref{almostc} holds. We will call this the almost-convexity of $(N\backslash U_\ve,d_{\rm can})$.

It is also possible to avoid using \eqref{quasiisom} as follows: using \eqref{4} and passing to the limit on $N\backslash U_\ve$
shows that $d_{\rm can}$ has a local H\"older bound there (with respect to $g_N$), and we conclude the localized almost convexity statement since the $g_N$-distance from $p$ to $\de U_{\ve '}$ is $O(\ve')$.

\subsection{Reduction of Theorem \ref{krf_gh} to Proposition \ref{claim4}}
\begin{proof}[Proof of Theorem \ref{krf_gh}]
From the volume form bound \eqref{volscal} we see that for any given $\delta>0$ there are $\ve=\ve(\delta)<\delta, T>0$ such that for all $t\geq T$ we have
\begin{equation}\label{stupid}
\frac{\mathrm{Vol}(\ti{U}_\ve,\omega(t))}{\mathrm{Vol}(M,\omega(t))}\leq\delta.
\end{equation}
We can also assume that $\ve$ is small enough so that the above-mentioned almost-convexity property of $N\backslash U_{\ve}$ holds, and we fix this value of $\ve(\delta)$ for the rest of the proof. Also, in all of the following, $\delta'>0$ will be a positive number that depends on $\delta$ and satisfies $\delta'(\delta)\to 0$ as $\delta\to 0$, which may change from line to line.\\

{\bf Claim 1. } For every $p\in U_\ve$ we have
\begin{equation}\label{triv}
d_Z(p,\de U_\ve)\leq\delta'.
\end{equation}

Recall here that $(Z,d_Z)$ denotes the metric completion of $(N^\circ,d_{\rm can})$.  Claim 1 follows easily from the upper bound for $\omega_{\rm can}$ in \eqref{quasiisom}, however we can also argue in a different way without using \eqref{quasiisom}, as follows. We employ the family of metrics $\omega_j$ in Proposition \ref{gh} that regularize $\omega_{\rm can}$ and have the property that $(N,\omega_j)\to (Z, d_Z)$ in the Gromov-Hausdorff sense. Thanks to Cheeger-Colding's extension of Colding's volume convergence theorem \cite[Theorem 5.9]{CC}, the volume noncollapsing bound in \eqref{3} implies that there is $C$ such that for all $x\in Z$ and $0<r<\mathrm{diam}(Z,d_Z)$ we have
\begin{equation}\label{balla1}
\mathcal{H}^{2n}(B^{d_Z}(x,r))\geq Cr^{2n},
\end{equation}
where here $\mathcal{H}^{2n}$ denotes the $2n$-dimensional Hausdorff measure. By definition we have an isometric embedding $\iota:(N^\circ,d_{\rm can})\hookrightarrow (Z,d_Z)$ and it is shown in \cite[p.758]{TZ2} that
\begin{equation}\label{meas1}
\mathcal{H}^{2n}(Z\backslash\iota(N^\circ))=0.
\end{equation}
On the other hand, on $N^\circ$ the renormalized limit measure $\nu$ is proportional to $\omega_{\rm can}^n$ by \cite[p.758]{TZ2}, and since this is proportional to $e^\vp f_*(\mathcal{M})$ with $\vp$ bounded, it follows that
\begin{equation}\label{meas2}
\int_{U_\ve\backslash D}\omega_{\rm can}^n\leq C\int_{\ti{U}_\ve}\omega_M^m\leq \delta',
\end{equation}
and so if we identify $U_\ve$ with its image in $Z$ under the homeomorphism $N\cong Z$, it follows from \eqref{meas1} and \eqref{meas2} that
\begin{equation}\label{balla2}
\mathcal{H}^{2n}(U_\ve)\leq \delta',
\end{equation}
and Claim 1 follows from \eqref{balla1} and \eqref{balla2}.\\

{\bf Claim 2. }We have
$$d_{\rm GH}((Z,d_Z),(N\backslash U_\ve,d_{\rm can}))\leq \delta', \quad {\rm where }\ \delta'(\delta)\to 0\ {\rm as }\ \delta\to 0.$$
We emphasize that here $(N\backslash U_\ve,d_{\rm can})$ denotes the restriction of the metric $d_{\rm can}$ from $N^\circ$ to the subset $N\backslash U_\ve$. However, by the almost-convexity property of $N\backslash U_\ve$, this differs from the distance induced by the metric $\omega_{\rm can}$ on $N\backslash U_\ve$ by at most $\delta$, so these two distances on $N\backslash U_\ve$ can be safely interchanged in our arguments.

To prove Claim 2 we use Claim 1 that allows us to define a map $F:Z\cong N\to N\backslash U_\ve$ (in general discontinuous) which is the identity on $N\backslash U_\ve$ and inside $U_\ve$ it maps $p$ to a point $q\in \de U_\ve$ with $d_Z(p,q)\leq\delta'$ (which is not unique, but we just choose any one of them). Let $G:N\backslash U_\ve\to N\cong Z$ denote the inclusion. It is elementary to check that $F$ and $G$ are a $3\delta'$-GH approximation, using the almost convexity property \eqref{almostc} and the fact that replacing $p$ by $q$ only distorts the $d_Z$-distance function by $\delta'$. For the reader's convenience we spell out the details, since a similar argument will also be used later: we need to show the following properties
\begin{equation}\label{1b}
d_Z(x,G(F(x)))\leq \delta', \quad x\in N,
\end{equation}
\begin{equation}\label{2b}
d_{\rm can}(y,F(G(y)))\leq \delta', \quad y\in N\backslash U_\ve,
\end{equation}
\begin{equation}\label{3b}
|d_Z(x,x')-d_{\rm can}(F(x),F(x'))|\leq 3\delta', \quad x,x'\in N,
\end{equation}
\begin{equation}\label{4b}
|d_{\rm can}(y,y')-d_Z(G(y),G(y'))|\leq \delta', \quad y,y'\in N\backslash U_\ve.
\end{equation}
Estimate \eqref{1b} is trivial when $x\in N\backslash U_\ve$, and follows from \eqref{triv} when $x\in U_\ve$. Estimate \eqref{2b} is trivial. Next, given any $x,x'\in N\backslash U_\ve$ using \eqref{almostc} we see that
\begin{equation}\label{5b}
d_Z(x,x')\leq d_{\rm can}(x,x')\leq d_Z(x,x')+\delta'.
\end{equation}
This immediately implies \eqref{4b}, so it remains to check \eqref{3b}, and for this we consider three cases. First, if $x,x'\in N\backslash U_\ve$ then \eqref{3b} follows from \eqref{5b}. Second, suppose $x\in N\backslash U_\ve, x'\in U_\ve.$ Then using the almost-convexity in \eqref{triv} we have
\begin{equation}\label{triv2}
d_Z(x',F(x'))\leq\delta',
\end{equation}
and using this and \eqref{5b} we obtain
$$d_Z(x,x')\leq d_Z(x,F(x'))+\delta'=d_Z(F(x),F(x'))+\delta'\leq d_{\rm can}(F(x),F(x'))+\delta',$$
and also
$$d_{\rm can}(F(x),F(x'))\leq d_Z(F(x),F(x'))+\delta'\leq d_Z(x,x')+d_Z(x',F(x'))+\delta'\leq d_Z(x,x')+2\delta',$$
proving \eqref{3b} in this case. Thirdly, suppose $x,x'\in U_\ve$, and use again \eqref{5b} and \eqref{triv2} to bound
\[\begin{split}
d_Z(x,x')&\leq d_Z(x,F(x))+d_Z(x',F(x'))+ d_Z(F(x),F(x'))\\
&\leq d_Z(F(x),F(x'))+2\delta'\leq d_{\rm can}(F(x),F(x'))+2\delta',
\end{split}\]
and
\[\begin{split}
d_{\rm can}(F(x),F(x'))&\leq d_Z(F(x),F(x'))+\delta'\\
&\leq d_Z(x,F(x))+d_Z(x',F(x'))+d_Z(x,x')+\delta'\leq d_Z(x,x')+3\delta',
\end{split}\]
completing the proof of \eqref{3b} and of  Claim 2.\\

Next, recall from \eqref{twyconv} that away from $S$ we have locally uniform convergence of $\omega(t)$ to $f^*\omega_{\rm can}$. Since $f:M^\circ\to N^\circ$ is a $C^\infty$ fiber bundle it follows easily that,
up to making $T$ larger, we have
\begin{equation}\label{claim2a}
d_{\rm GH}((N\backslash U_\ve,\omega_{\rm can}),(M\backslash\ti{U}_\ve,\omega(t)))\leq \delta,
\end{equation}
for all $t\geq T$, see e.g. \cite[Theorem 5.23]{To}. But as a consequence of almost-convexity, the distance function given by $(N\backslash U_\ve,\omega_{\rm can})$ differs from $(N\backslash U_\ve,d_{\rm can})$ by at most $\delta$,  so we also have
\begin{equation}\label{claim2}
d_{\rm GH}((N\backslash U_\ve,d_{\rm can}),(M\backslash\ti{U}_\ve,\omega(t)))\leq 2\delta,
\end{equation}
for all $t\geq T$.
Lastly, we have the following claim:\\

{\bf Claim 3. }Up to making $T$ larger, we have
$$d_{\rm GH}((M\backslash\ti{U}_\ve,\omega(t)),(M,\omega(t)))\leq   \delta', \quad {\rm where }\ \delta'(\delta)\to 0\ {\rm as }\ \delta\to 0,$$
for all $t\geq T$.\\

Combining Claims 2 and 3 with \eqref{claim2} we conclude that
$$(M,\omega(t))\to (Z,d_Z),$$
in the Gromov-Hausdorff topology as $t\to \infty$, which will complete the proof of Theorem \ref{krf_gh}.

The proof of Claim 3 relies heavily on the following statement, which can be thought of as  an almost-convexity of $(M\backslash \ti{U}_\ve,\omega(t))$ inside $(M,\omega(t))$ uniformly in $t\geq T$. Denote by $d_t$ the distance function of $(M,\omega(t))$ and by $\hat{d}_t$ the distance function of $(M\backslash\ti{U}_\ve,\omega(t))$, so that we trivially have $d_t(x,x')\leq \hat{d}_t(x,x')$ for all $x,x'\in M\backslash \ti{U}_\ve$. Then we have:

\begin{prop}\label{claim4}
Given $\delta>0$ there are $\delta',T>0$, with $\delta'(\delta)\to 0$ as $\delta\to 0$, such that for all $x,x'\in M\backslash \ti{U}_\ve$ and all $t\geq T$ we have
\begin{equation}\label{5}
d_t(x,x')\leq \hat{d}_t(x,x')\leq d_t(x,x')+\delta'.
\end{equation}
\end{prop}

Indeed, assuming Proposition \ref{claim4}, the proof of Claim 3 is completely analogous to the proof of Claim 2, and we briefly outline it. First, we have the analog of Claim 1, namely that, up to enlarging $T$, for all $t\geq T$ and $x\in \ti{U}_\ve$ we have
\begin{equation}\label{glenbenton}
d_t(x,\de \ti{U}_\ve)\leq   \delta', \quad {\rm where }\ \delta'(\delta)\to 0\ {\rm as }\ \delta\to 0.
\end{equation}
To see this we use the volume estimates \eqref{volscal} and \eqref{volbd}
which imply that for all $x\in M,0<r<\mathrm{diam}(M,g(t)),t\geq 0,$
$$\frac{\mathrm{Vol}(B^{g(t)}(x,r),\omega(t))}{\mathrm{Vol}(M,\omega(t))}\geq C^{-1}r^{2n},$$
and so \eqref{glenbenton} follows from this together with \eqref{stupid}.

Using \eqref{glenbenton}, for each $t\geq T$ we define a discontinuous map $F_t:M\to M\backslash\ti{U}_\ve$ which is the identity on $M\backslash \ti{U}_\ve$ and inside $U_\ve$ it maps $p$ to some point $q\in \de\ti{U}_\ve$ with $d_t(p,q)\leq\delta'$. One defines then $G:M\backslash\ti{U}_\ve\to M$ to be the inclusion, and using \eqref{5} one checks exactly as in Claim 2 that $F_t$ and $G$ give a $3\delta'$-GH approximation between $(M,d_t)$ and $(M\backslash\ti{U}_\ve,\hat{d}_t)$, thus proving Claim 3. The proof of Theorem \ref{krf_gh} is thus reduced to proving Proposition \ref{claim4}.

\subsection{Proof of Proposition \ref{claim4}}
The only nontrivial inequality to prove is $\hat{d}_t(x,x')\leq d_t(x,x')+\delta'$. For this, we first observe that given any $p,q\in M\backslash\ti{U}_\ve$ we know from the almost-convexity property of $(N\backslash U_\ve, d_{\rm can})$ that their images $f(p),f(q)\in N\backslash U_\ve$ can be joined by a path
$\gamma$ in $N\backslash U_\ve$ with
$$L_{g_{\rm can}}(\gamma)\leq d_{\rm can}(f(p),f(q))+\delta'.$$
Since $f$ is a smooth fiber bundle over $N\backslash U_\ve$, we can easily find a path $\ti{\gamma}$ in $M\backslash \ti{U}_\ve$ joining $p$ and $q$ with $f\circ\ti{\gamma}=\gamma$, see e.g. \cite[Theorem 5.23]{To}. Thanks to the uniform convergence in \eqref{twyconv}, and the fact that $d_{\rm can}(f(p),f(q))\leq C$ for some $C$ independent of $p,q,$ we see that (up to increasing $T$) for all $t\geq T$ we have
$$\hat{d}_t(p,q)\leq L_{g_t}(\ti{\gamma})\leq L_{f^*g_{\rm can}}(\ti{\gamma})+\delta'=L_{g_{\rm can}}(\gamma)+\delta'\leq d_{\rm can}(f(p),f(q))+2\delta',$$
so to complete the proof of \eqref{5} we are left with proving the following:\\

{\bf Claim 4. }Up to making $T$ larger, we have
\begin{equation}\label{finalgoal}
d_{\rm can}(f(p),f(q))\leq d_t(p,q)+\delta',
\end{equation}
for all $t\geq T$ and all $p,q\in M\backslash\ti{U}_\ve$.\\

We will sometimes tacitly replace $\delta'$ by $C\delta'$, and without loss we may assume that $d_{\rm can}(f(p),f(q))\geq \delta'$.
The rough idea to prove Claim 4 is to first replace the distance by a version of Perelman's reduced distance, and then use a smearing argument to show these two are roughly the same.

First, we shall reparametrize the flow. Let $T$ be a given large time, whose precise value will be determined at the end of the argument. Recall that our K\"ahler metrics satisfy the K\"ahler-Ricci flow
$$\frac{\de}{\de t}\omega(t)=-\Ric(\omega(t))-\omega(t),\quad\omega(0)=\omega_M.$$
If as usual we let $g(t)$ denote their associated Riemannian metrics, then the Riemannian metrics
$$\ti{g}(s)=e^{t-T}g(t-T), \quad s=\frac{1}{2}(e^{t-T}-1),$$
solve the standard Ricci flow
$$\frac{\de}{\de s}\ti{g}(s)=-2\Ric(\ti{g}(s)),\quad s\geq s_0:=\frac{1}{2}(e^{-T}-1),$$
with $\ti{g}(0)=g(T),$
and we can convert back from $\ti{g}(s)$ to $g(t-T)$ by
$$g(t-T)=\frac{\ti{g}(s)}{1+2s},\quad t=T+\log(1+2s).$$
The scalar curvature bound \eqref{stscal} translates to
\begin{equation}\label{stscal2}
\sup_M |R(\ti{g}(s))|\leq \frac{C}{1+2s},
\end{equation}
for all $s\geq s_0$.

As in \cite{Pe}, we let $\tau=-s$, and letting $\ti{g}(\tau)$ be the metrics $\ti{g}(s)$ with parameter $s=-\tau$, then these solve the backwards Ricci flow
\[
\frac{\de}{\de \tau}\ti{g}= 2\Ric(\ti{g}),\quad \ti{g}\big|_{\tau=0}=g(T).
\]
We will work with $0\leq\tau\leq\bar{\tau}\ll 1$, where the choice of $\bar{\tau}$ depends on $\delta$, to be specified. In particular, since $\ov{\tau}$ is small, it follows from \eqref{twyconv} that  $\ti{g}(\tau)$ is uniformly close to $f^*g_{\rm can}$ on $M\backslash \ti{U}_\ve$ for $0\leq\tau\leq\ov{\tau}$.

Following Perelman \cite[\S 7]{Pe} the $\mathcal{L}$-length of a curve $\gamma(\tau)$ in spacetime is defined by
$$\mathcal{L}(\gamma)=\int \sqrt{\tau} (R(\ti{g}(\tau))+ |\partial_\tau \gamma|^2_{\ti{g}(\tau)}) d\tau.$$
The $\mathcal{L}$-distance between two points in spacetime is the infimum of such, and given $p,q\in M\backslash\ti{U}_\ve$
following Perelman we will denote by $L(q,\ov{\tau})$ the $\mathcal{L}$-distance between $(p,\tau=0)$ and $(q, \tau=\bar{\tau})$.

To start, using the almost-convexity of $(N\backslash U_\ve,d_{\rm can})$ we can
join $f(p)$ and $f(q)$ by a path $\gamma$ inside $N\backslash U_\ve$ with length
$$L_{g_{\rm can}}(\gamma)\leq d_{\rm can}(f(p),f(q))+\delta',$$
and let $\ti{\gamma}$ be a lift to a path in $M\backslash \ti{U}_\ve$ joining $p$ and $q$.

We then parametrize $\ti{\gamma}$ by $\tau\in [0,\ov{\tau}]$ so that $|\de_\tau\ti{\gamma}|_{\ti{g}(\tau)}=\frac{A}{2\sqrt{\tau\ov{\tau}}}$ for all $0\leq\tau\leq\ov{\tau}$, where $$A=\int_0^{\ov{\tau}}|\de_\tau\ti{\gamma}|_{\ti{g}(\tau)}d\tau
\leq d_{\rm can}(f(p),f(q))+\delta',$$
using that $\ti{g}(\tau)$ is close to $f^*g_{\rm can}$ on $M\backslash \ti{U}_\ve$ for $0\leq\tau\leq\ov{\tau}$, and that $d_{\rm can}(f(p),f(q))\leq C$ by the diameter bound for $(N^\circ,d_{\rm can})$. Then, using the scalar curvature bound \eqref{stscal2}, we can estimate
\[\begin{split}
L(q,\ov{\tau}) &\leq\mathcal{L}(\ti{\gamma})=\int_0^{\ov{\tau}}\sqrt{\tau}((R(\ti{g}(\tau))+|\de_\tau\ti{\gamma}|^2_{\ti{g}(\tau)})d\tau
\leq C\ov{\tau}^{\frac{3}{2}}+\int_0^{\ov{\tau}}\sqrt{\tau}|\de_\tau\ti{\gamma}|^2_{\ti{g}(\tau)}d\tau\\
&=
C\ov{\tau}^{\frac{3}{2}}+\frac{A^2}{2\sqrt{\ov{\tau}}}\\
&\leq C\ov{\tau}^{\frac{3}{2}}+\frac{1}{2\sqrt{\ov{\tau}}}(d_{\rm can}(f(p),f(q))+\delta')^2,
\end{split}\]
and we can absorb the term with $\ov{\tau}^{\frac{3}{2}}$ into the term with $\delta'$ by choosing $\bar{\tau}\ll \delta'$, thus obtaining
\begin{equation}\label{Lupperbd}
L(q,\ov{\tau})\leq \mathcal{L}(\ti{\gamma}) \leq  \frac{1}{2\bar{\tau}^{1/2}} (d_{\rm can}(f(p),f(q))+C\delta')^2.
\end{equation}
The same is true if $q$ is replaced by any point $q'$ with $f(q')\in B^{g_{\rm can}}(f(q),\delta')$.

We wish to show the almost matching lower bound
\begin{equation}\label{Llowerbd}
L(q,\ov{\tau})\geq \frac{1}{2\bar{\tau}^{1/2}} (d_{\rm can}(f(p),f(q))-C\delta')^2.
\end{equation}
By the triangle inequality, and up to a small modification of $\bar{\tau}$ to $\bar{\tau}(1+O(\delta'))$, it is enough to prove
\begin{equation}\label{Llowerbd2}
L(q',\ov{\tau})\geq \frac{1}{2\bar{\tau}^{1/2}} (d_{\rm can}(f(p),f(q))-C\delta')^2
\end{equation}
for some  $q'$ with $f(q')\in B^{g_{\rm can}}(f(q),\delta')$. The main enemy is that the $\mathcal{L}$-geodesics can go into the region $\ti{U}_\ve$ where we do not have much control of the metric.

On any minimal $\mathcal{L}$-geodesic $\gamma$ from $(p,0)$ to $(q', \bar{\tau})$, thanks to \eqref{Lupperbd} (applied with $q$ replaced by $q'$) and the diameter bound for $(N^\circ,d_{\rm can})$ we see that
\begin{equation}\label{Lupper}
\mathcal{L}(\gamma)=L(q',\ov{\tau})\leq C\bar{\tau}^{-1/2}(d_{\rm can}(f(p),f(q'))+C\delta')^2\leq C\bar{\tau}^{-1/2}.
\end{equation}
Fix a $g_{\rm can}$-ball $B\Subset N\backslash U_\ve$ centered at $f(p)$ of some radius $r>0$ (which depends only on $\ve$), let $\ti{B}=f^{-1}(B)$, and let $0<\ov{\tau}'\leq\ov{\tau}$ be the minimum between $\frac{1}{2}\ov{\tau}$ and the first time when the curve $\gamma$ exits $\ti{B}$. Since $\ti{g}(\tau)$ is uniformly close to $f^*g_{\rm can}$ along $\gamma(\tau)$ for $0\leq\tau\leq \ov{\tau}'$, using \eqref{Lupper} we have
\[\begin{split}
r&\leq C\int_0^{\ov{\tau}'}|\de_\tau\gamma|_{\ti{g}(\tau)}d\tau\\
&\leq C\left(\int_0^{\ov{\tau}'}\sqrt{\tau}|\de_\tau\gamma|^2_{\ti{g}(\tau)}d\tau\right)^{\frac{1}{2}}\left(\int_0^{\ov{\tau}'}\frac{1}{\sqrt{\tau}}d\tau\right)^{\frac{1}{2}}\\
&\leq C\ov{\tau}'^{\frac{1}{4}}\left(C\ov{\tau}'^{\frac{3}{2}}+\int_0^{\ov{\tau}'}\sqrt{\tau}(R(\ti{g}(\tau))+|\de_\tau\gamma|^2_{\ti{g}(\tau)})d\tau\right)^{\frac{1}{2}}\\
&\leq C\ov{\tau}'^{\frac{1}{4}}\left(C\ov{\tau}'^{\frac{3}{2}}+C\bar{\tau}^{-\frac{1}{2}}\right)^{\frac{1}{2}}\\
&\leq C\ov{\tau}'^{\frac{1}{4}}\ov{\tau}^{-\frac{1}{4}},
\end{split}\]
i.e.
\begin{equation}\label{lowertau}
\ov{\tau}'\geq C^{-1}\ov{\tau}.
\end{equation}

Perelman \cite[\S 7.1]{Pe} showed that
\begin{equation}\label{jacob}
\tau^{-m} \exp(  - l(\tau) ) J(\tau)
\end{equation}
is nonincreasing in $\tau$ along an $\mathcal{L}$-geodesic, where $l(q,\tau)=\frac{1}{2\sqrt{\tau}}L(q,\tau)$ is the reduced length and $J$ is the Jacobian of the $\mathcal{L}$-exponential. Thanks to \eqref{Lupper} we have
\begin{equation}\label{reducedl}
l(q',\ov{\tau})\leq \frac{C}{\ov{\tau}}.
\end{equation}
Thus, Perelman's monotonicity together with \eqref{lowertau} and \eqref{reducedl} gives that for $\ov{\tau}'\leq\tau\leq\ov{\tau}$ we have
\begin{equation}\label{jacob2}
J(\tau)\geq \left(\frac{\tau}{\ov{\tau}}\right)^me^{l(\tau)-l(\ov{\tau})}J(\ov{\tau})\geq C^{-1}e^{-\frac{C}{\ov{\tau}}}J(\ov{\tau}),
\end{equation}
on $M$.

Consider the set $\Gamma$ of all the minimal $\mathcal{L}$-geodesics from $(p,0)$ to $(q', \bar{\tau})$ with $q'$ such that $f(q')\in B^{g_{\rm can}}(f(q),\delta')$, and consider the subset $E\subset [\ov{\tau}',\ov{\tau}]\times M$ defined by
$$E=\bigcup_{\gamma\in\Gamma}\{(\tau,\gamma(\tau))\ |\ \bar{\tau}'\leq \tau\leq \bar{\tau}\}.$$
Writing $E_\tau=E\cap(\{\tau\}\times M)$ (viewed as a subset of $M$), the spacetime volume of the region $E$ is defined by
$$\mathrm{Vol}(E):=\int_{\ov{\tau}'}^{\ov{\tau}}\int_{E_\tau}\ti{\omega}(\tau)^m d\tau.$$
Let $\mathcal{L}\mathrm{exp}_{p,\tau}:T_pM\to M$ be the $\mathcal{L}$-exponential map based at $p$ with parameter $\tau$. Then $E_\tau=\mathcal{L}\mathrm{exp}_{p,\tau}(F)$ where $F\subset T_pM$ is a $\tau$-independent open subset, and up to sets of measure zero $\mathcal{L}\mathrm{exp}_{p,\tau}$ is a diffeomorphism between $F$ and $E_\tau$, see the discussion in \cite[\S 17]{KL} or \cite[\S 8]{CCGGIIKLLN}.
Equipping $T_pM$ with the Euclidean metric $\ti{g}_p(0)$, and letting $dv$ be its volume element, we can write
$$\int_{E_\tau}\ti{\omega}(\tau)^m=\int_{F}J(\tau) dv,$$
and thanks to \eqref{jacob2}, for all $0\leq\tau\leq\ov{\tau}$ we can estimate
$$\int_{F}J(\tau) dv\geq C^{-1}e^{-\frac{C}{\ov{\tau}}}\int_{F}J(\ov{\tau}) dv=
C^{-1}e^{-\frac{C}{\ov{\tau}}}\int_{E_{\ov{\tau}}}\ti{\omega}(\ov{\tau})^m,$$
but up to sets of measure zero, $E_{\ov{\tau}}$ equals $f^{-1}(B^{g_{\rm can}}(f(q),\delta'))$, whose volume with respect to $\omega_M^m$ is at least $C^{-1}\delta'^{2n}$. Using the volume form bound \eqref{volscal} we thus see that
$$\int_{E_\tau}\ti{\omega}(\tau)^m\geq C^{-1}e^{-\frac{C}{\ov{\tau}}}\delta'^{2n}e^{-(m-n)T},$$
and using that $\ov{\tau}'\leq\frac{1}{2}\ov{\tau}$ by definition, we conclude that
\begin{equation}\label{vole}
\mathrm{Vol}(E)\geq C^{-1}\ov{\tau}e^{-\frac{C}{\ov{\tau}}}\delta'^{2n}e^{-(m-n)T}.
\end{equation}

Next, for $0<\eta\ll 1$, to be chosen later depending on $\delta',\ov{\tau}$, and for $\ov{\tau}'\leq\tau\leq\ov{\tau}$, we have using \eqref{volscal}
\[\begin{split}
\int_{\ti{U}_\eta}\ti{\omega}(\tau)^m&\leq C\int_{\ti{U}_\eta}\omega(T)^m\leq Ce^{-(m-n)T}\int_{\ti{U}_\eta}\omega_M^m\\
&\leq C\eta e^{-(m-n)T},
\end{split}\]
where we used that the $g_M$-volume of $\ti{U}_\eta$ is at most $C\eta$.
Thus the set given by $(\tau,\gamma(\tau))$ with $\gamma(\tau)\in \ti{U}_\eta$, which is a subset of $E$, has spacetime volume bounded above by
$$C\ov{\tau}\eta e^{-(m-n)T}\leq C\eta \delta'^{-2n}e^{C/\bar{\tau}}\mathrm{Vol}(E),$$
using \eqref{vole}.
In particular, on a typical minimal $\mathcal{L}$-geodesic in $\Gamma$, the $\tau$-time spent inside $\ti{U}_\eta$ is less than $C\bar{\tau}\eta \delta'^{-2n}e^{C/\bar{\tau}}$. For each such $\mathcal{L}$-geodesic $\gamma$ we split $[0,\ov{\tau}]$ into the subset $I$ defined by the property that  $\tau\in I\Leftrightarrow \gamma(\tau)\in \ti{U}_\eta$, and its complement $J=[0,\ov{\tau}]\backslash I$.
Then we have
$$|I|\leq C\bar{\tau}\eta \delta'^{-2n}e^{C/\bar{\tau}},$$ and thanks to \eqref{lowertau} we know that every $\tau\in I$ satisfies $\tau\geq C^{-1}\ov{\tau}.$ The same argument that we used to prove \eqref{Lupper} shows that
\begin{equation}\label{Lupper2}
\mathcal{L}(\gamma)\leq C\ov{\tau}^{-\frac{1}{2}}.
\end{equation}
Splitting
$$\int_0^{\bar{\tau} } |\partial_\tau \gamma|_{\ti{g}(\tau)}  d\tau=\int_I |\partial_\tau \gamma|_{\ti{g}(\tau)}  d\tau+\int_J |\partial_\tau \gamma|_{\ti{g}(\tau)}  d\tau,$$
we can then estimate using \eqref{stscal2}, \eqref{lowertau} and \eqref{Lupper2}
\begin{equation}\label{ggg1}
\begin{split}
\int_I |\partial_\tau \gamma|_{\ti{g}(\tau)}  d\tau&\leq C\ov{\tau}^{-\frac{1}{4}}\int_I \tau^{\frac{1}{4}}|\partial_\tau \gamma|_{\ti{g}(\tau)}  d\tau\\
&\leq C\ov{\tau}^{-\frac{1}{4}}\left(\int_I \sqrt{\tau}|\partial_\tau \gamma|^2_{\ti{g}(\tau)}  d\tau\right)^{\frac{1}{2}}|I|^{\frac{1}{2}}\\
&\leq C\ov{\tau}^{-\frac{1}{4}}\bar{\tau}^{\frac{1}{2}}\eta^{\frac{1}{2}} \delta'^{-n}e^{C/\bar{\tau}}
\left(C\ov{\tau}^{\frac{3}{2}}+\int_I \sqrt{\tau}(R(\ti{g}(\tau))+|\partial_\tau \gamma|^2_{\ti{g}(\tau)})  d\tau\right)^{\frac{1}{2}}\\
&\leq C\ov{\tau}^{\frac{1}{4}}\eta^{\frac{1}{2}} \delta'^{-n}e^{C/\bar{\tau}}\left(C\ov{\tau}^{\frac{3}{2}}+
C\ov{\tau}^{-\frac{1}{2}}\right)^{\frac{1}{2}}\\
&\leq  C\eta^{\frac{1}{2}} \delta'^{-n}e^{C/\bar{\tau}},
\end{split}\end{equation}
and combining this estimate with \eqref{schwarz} we see that the $g_N$-distance traversed by $f(\gamma)$ inside $U_\eta$ is bounded above by
\begin{equation}\label{ggg2}
C\eta^{\frac{1}{2}} \delta'^{-n}e^{C/\bar{\tau}}\ll \eta^{\beta},
\end{equation}
for small enough $\eta$ and some fixed small positive exponent $\beta<\frac{1}{2}$.

The issue now is how to use these bounds to estimate from above the $g_{\rm can}$-length of the curve $f(\gamma)$. Outside of $\ti{U}_\eta$ we have that $\ti{g}(\tau)$ is uniformly close to $f^*g_{\rm can}$ (up to enlarging $T$, depending on our choice of $\eta$, which itself was chosen depending ultimately on $\delta$), so we have
\begin{equation}\label{ggg3}
\int_J |\partial_\tau \gamma|_{\ti{g}(\tau)}  d\tau\geq \int_J |\de_\tau f(\gamma)|_{g_{\rm can}}d\tau -C\delta',
\end{equation}
which will give us the desired bound for the $g_{\rm can}$-length of the portion of $f(\gamma)$ outside $U_\eta$ (i.e. when $\tau\in J$).
On the other hand, to estimate the the  $g_{\rm can}$-length of the portion of $f(\gamma)$ inside $U_\eta$ (i.e. when $\tau\in I$), we employ the metric bounds \eqref{quasiisom} as follows.

Consider the following events: $f(\gamma)$ enters $U_{\eta}$ and reaches $U_{\eta/2}$ before returning to the boundary of $U_{\eta}$. On the one hand, by definition, the $g_N$-distance traversed during this whole event is at least $\eta$, and on the other hand we have shown in \eqref{ggg2} that it is also less than $\eta^{\beta}$, the discussion is local: the event takes place in a local coordinate chart exhibiting $D$ as a simple normal crossings divisor. Let
$f(\gamma(\tau_{\rm entry}))$  and $f(\gamma(\tau_{\rm exit}))$ be the entry and exit points of one event, so that their $g_{\rm can}$-distance is at most
$$\int_{\tau_{\rm entry}}^{\tau_{\rm exit}}|\partial_\tau f(\gamma)|_{g_{\rm can}}d\tau.$$
By the explicit control \eqref{quasiisom} on $g_{\rm can}$, we have that
\begin{equation}\label{quasi}
C^{-1}\omega_{\rm cone}\leq \omega_{\rm can}\leq C\left(1-\sum_{i=1}^\mu\log|s_i|_{h_i}\right)^A \omega_{\rm cone}.
\end{equation}
For simplicity suppose first that $D$ has only one component, which in our local chart is given by $\{z_1=0\}$. Then we can assume without loss that in this chart the boundary of $U_\eta$ is given by $\{|z_1|=\eta\}$, and in our chart \eqref{quasi} reads
\[
C^{-1}\left(\frac{idz_1\wedge d\ov{z}_1}{|z_1|^{2(1-\gamma)}}+\sum_{j=2}^nidz_j\wedge d\ov{z}_j\right)\leq\omega_{\rm can}\leq C\left(1-\log|z_1|\right)^A \left(\frac{idz_1\wedge d\ov{z}_1}{|z_1|^{2(1-\gamma)}}+\sum_{j=2}^nidz_j\wedge d\ov{z}_j\right).
\]
The entry and exit points are both on $\{|z_1|=\eta\}$, have $g_{\rm can}$-distance at most $\int_{\tau_{\rm entry}}^{\tau_{\rm exit}}|\partial_\tau f(\gamma)|_{g_{\rm can}}d\tau$, and hence their $g_{\rm cone}$-distance is at most $C$ times that. Therefore there exists another path joining these entry and exit points, which is contained in the boundary of $U_{\eta}$ (in particular, it does not come into $U_{\eta/2}$) and whose
$g_{\rm cone}$-length is also at most $C\int_{\tau_{\rm entry}}^{\tau_{\rm exit}}|\partial_\tau f(\gamma)|_{g_{\rm can}}d\tau$, and hence whose $g_{\rm can}$-length is bounded above by
\[
C|\log \eta|^C \int_{\tau_{\rm entry}}^{\tau_{\rm exit}}|\partial_\tau \gamma|_{\ti{g}(\tau)}d\tau.
\]
The general case when in our chart we see several components of $D$ is dealt with similarly.
We perform this construction for all the events (which are disjoint). This gives a replacement $\gamma'$ of $f(\gamma)$,  staying outside $U_{\eta/2}$, agreeing with $f(\gamma)$ between the events, and whose $g_{\rm can}$-length traversed in each event is at most  $C|\log \eta|^C$ times the corresponding integral of $|\partial_\tau \gamma|_{\ti{g}(\tau)}$. Thus, using \eqref{ggg1}, \eqref{ggg2} and \eqref{ggg3},
\begin{equation}\label{ggg4}\begin{split}
d_{\rm can}(f(p),f(q))&\leq L_{g_{\rm can}}(\gamma') \leq \int_J|\de_\tau f(\gamma)|_{g_{\rm can}}d\tau+ C|\log \eta|^C \int_I|\partial_\tau \gamma|_{\ti{g}(\tau)}d\tau\\
&\leq
\int_J |\partial_\tau \gamma|_{\ti{g}(\tau)}  d\tau+C\delta'+ C|\log \eta|^C \eta^{\beta}\\
&\leq \int_0^{\ov{\tau}} |\partial_\tau \gamma|_{\ti{g}(\tau)}  d\tau+C\delta',
\end{split}
\end{equation}
choosing $\eta$ small enough. Since $\gamma$ here is a minimal $\mathcal{L}$-geodesic from $(p,0)$ to $(q',\ov{\tau})$, arguing as before we see that
\begin{equation}\label{ggg5}\begin{split}
\int_0^{\ov{\tau}} |\partial_\tau \gamma|_{\ti{g}(\tau)}  d\tau&\leq \left(\int_0^{\ov{\tau}}\sqrt{\tau}|\de_\tau\gamma|^2_{\ti{g}(\tau)}d\tau\right)^{\frac{1}{2}}\left(\int_0^{\ov{\tau}}\frac{1}{\sqrt{\tau}}d\tau\right)^{\frac{1}{2}}\\
&\leq \sqrt{2}\ov{\tau}^{\frac{1}{4}}\left(C\ov{\tau}^{\frac{3}{2}}+\int_0^{\ov{\tau}}\sqrt{\tau}(R(\ti{g}(\tau))+|\de_\tau\gamma|^2_{\ti{g}(\tau)})d\tau\right)^{\frac{1}{2}}\\
&=\sqrt{2}\ov{\tau}^{\frac{1}{4}}\left(C\ov{\tau}^{\frac{3}{2}}+L(q',\ov{\tau})\right)^{\frac{1}{2}},
\end{split}
\end{equation}
and so from \eqref{ggg4} and \eqref{ggg5} (again we can absorb the term with $\ov{\tau}^{\frac{3}{2}}$) we get
\[
L(q',\ov{\tau})\geq \frac{1}{2\bar{\tau}^{1/2}} (d_{\rm can}(f(p),f(q))-C\delta')^2
\]
as desired, for some $q'$ with $f(q')\in B^{g_{\rm can}}(f(q),\delta')$, which establishes \eqref{Llowerbd} and \eqref{Llowerbd2}.
As explained above, up to small modification to $\bar{\tau}$ this implies that the same statement \eqref{Llowerbd2} holds for all $q'$ with $f(q')\in B^{g_{\rm can}}(f(q),\delta')$.

The main difference between this statement and Claim 4 is that we wish to compute distance with respect to a fixed time metric, rather than evolving metrics. Again following \cite{Pe} we  consider $\bar{L}(q',\tau)= 2\sqrt{\tau} L(q',\tau)$. As $\tau\to 0^+$, the function $\bar{L}$ tends to  $d_{\ti{g}(0)}(p,q')^2$
(see \cite[Lemma 7.47]{CCGGIIKLLN}), and according to \cite[(7.15)]{Pe} we have
\begin{equation}\label{ddt}
\left(\frac{\partial}{\partial \tau}+\Delta_{\ti{g}(\tau)}\right)\bar{L} \leq 4m.
\end{equation}
Recall that \eqref{Llowerbd} gives
\begin{equation}\label{Lbar2}
\bar{L}(q', \bar{\tau}) \geq (d_{\rm can}(f(p),f(q))-C\delta')^2,
\end{equation}
for $q'$ with $f(q')\in B^{g_{\rm can}}(f(q),\delta')$, and \eqref{Lupperbd} gives
\begin{equation}\label{Lbar}
\bar{L}(q',\tau)\leq C,
\end{equation}
for all $0<\tau\leq \bar{\tau}$.

Let $\chi$ be a smooth (time-independent) cutoff function on $N$ supported in $B^{g_{\rm can}}(q,\delta')$ and equal to $1$ on $B^{g_{\rm can}}(q,\delta'/2)$, and denote by the same symbol its pullback to $M$ via $f$. Then by \eqref{schwarz} we have
\begin{equation}\label{Delta}
\sup_M|\Delta_{\ti{g}(\tau)} \chi|\leq C\delta'^{-2},
\end{equation}
for $0\leq \tau\leq\ov{\tau}$.  Integrating  the $\tau$-time evolution of $\int_M \chi \bar{L}(\cdot,\tau) \ti{\omega}^m(\tau)$ with respect to $\tau\in[0,\ov{\tau}]$, and using \eqref{ddt}, we obtain
\[\begin{split}
&\int_M \chi \bar{L}(\cdot,0)\ti{\omega}^m(0)\\
&=\int_M \chi \bar{L}(\cdot,\ov{\tau})\ti{\omega}^m(\bar{\tau}) - \int_0^{\bar{\tau}} \int_M  \chi\left(\frac{\de}{\de \tau}\bar{L}(\cdot,\tau)\right)\ti{\omega}^m(\tau)  d\tau -2\int_0^{\ov{\tau}}\int_M \chi\bar{L}(\cdot,\tau) R(\ti{g}(\tau))\ti{\omega}(\tau)^md\tau\\
  &\geq\int_M \chi \bar{L}(\cdot,\ov{\tau})\ti{\omega}^m(\bar{\tau}) + \int_0^{\bar{\tau}} \int_M  \bar{L}(\cdot,\tau) \Delta_{\ti{g}(\tau)} \chi\ti{\omega}^m(\tau) d\tau
  -4m \int_0^{\bar{\tau}} \int_M  \bar{L}(\cdot,\tau) \chi\ti{\omega}^m(\tau) d\tau \\
  &-2\int_0^{\ov{\tau}}\int_M \chi\bar{L}(\cdot,\tau) R(\ti{g}(\tau))\ti{\omega}(\tau)^md\tau,
\end{split}
\]
and employing \eqref{volscal}, \eqref{stscal2}, \eqref{Lbar} and \eqref{Delta} we can bound
\[\begin{split}
\int_M  \bar{L}(\cdot,\tau) \Delta_{\ti{g}(\tau)} \chi\ti{\omega}^m(\tau)&\geq -C\delta'^{-2}\int_{f^{-1}(B^{g_{\rm can}}(q,\delta'))}\ti{\omega}^m(\tau)\\
&\geq -C\delta'^{-2}e^{-(m-n)T}\int_{f^{-1}(B^{g_{\rm can}}(q,\delta'))}\omega_M^m\\
&\geq -Ce^{-(m-n)T}\delta'^{2n-2},
\end{split}\]
as well as
\[\begin{split}
-4m \int_M  \bar{L}(\cdot,\tau) \chi\ti{\omega}^m(\tau)&\geq -C\int_{f^{-1}(B^{g_{\rm can}}(q,\delta'))}\ti{\omega}^m(\tau)\\
&\geq -Ce^{-(m-n)T}\int_{f^{-1}(B^{g_{\rm can}}(q,\delta'))}\omega_M^m\\
&\geq -Ce^{-(m-n)T}\delta'^{2n},
\end{split}\]
and similarly
\[ -2\int_M \chi\bar{L}(\cdot,\tau) R(\ti{g}(\tau))\ti{\omega}(\tau)^m\geq -Ce^{-(m-n)T}\delta'^{2n},\]
and so, using also \eqref{Lbar2},
\[\begin{split}
& \int_M \chi \bar{L}(\cdot,0)\ti{\omega}^m(0)
\\
& \geq  \int_M\chi \bar{L}(\cdot,\ov{\tau})\ti{\omega}^m(\bar{\tau}) -C\bar{\tau}\delta'^{2n-2} e^{-(m-n)T}
\\
& \geq  (d_{\rm can}(f(p),f(q))-C\delta')^2\int_M \chi \ti{\omega}^m(\ov{\tau}) -C\bar{\tau}\delta'^{2n-2} e^{-(m-n)T}\\
&\geq  \left((d_{\rm can}(f(p),f(q))-C\delta')^2-C\bar{\tau} \delta'^{-2}\right)\int_M \chi \ti{\omega}^m(\ov{\tau}),
\end{split}
\]
where in the last line we used that $\int_M\chi\ti{\omega}(\tau)^m\geq C^{-1}\delta'^{2n}e^{-(m-n)T}$, which again comes from \eqref{volscal}. By
choosing $\bar{\tau}\leq C^{-1} \delta'^3$ we can ignore the term with $\bar{\tau} \delta'^{-2}$.

Now, integrating  the $\tau$-time evolution of $\int_M \chi \ti{\omega}^m(\tau)$ with respect to $\tau\in[0,\ov{\tau}]$ we obtain
$$\int_M \chi \ti{\omega}^m(\ov{\tau})-\int_M \chi \ti{\omega}^m(0)=2\int_0^{\ov{\tau}}\chi R(\ti{g}(\tau))\ti{\omega}(\tau)^m d\tau
\geq -C\ov{\tau}\delta'^{-2}e^{-(m-n)T},$$
and so
\begin{equation}\label{gr1}
\begin{split}
& \int_M \chi \bar{L}(\cdot,0)\ti{\omega}^m(0)
\\
&\geq (d_{\rm can}(f(p),f(q))-C\delta')^2\int_M \chi \ti{\omega}^m(0)-C\ov{\tau}\delta'^{-2}e^{-(m-n)T}.
\end{split}
\end{equation}

Now using $C^0$ metric convergence in the regular region, the $g(T)$-distance between $q$ and $q'$ is bounded by $C\delta'$ in the support of $\chi$, so
\begin{equation}\label{gr2}
\begin{split}
& \int_M \chi \bar{L}(\cdot,0)\ti{\omega}^m(0)
\\
&
= \int_M \chi d_{\ti{g}(0)}(p,\cdot)^2\ti{\omega}^m(0)=\int_M \chi d_{g(T)}(p,\cdot)^2\ti{\omega}^m(0)
\\
& \leq (d_{g(T)}(p,q)+C\delta' )^2 \int_M \chi  \ti{\omega}^m(0).
\end{split}
\end{equation}
Combining \eqref{gr1} and \eqref{gr2} and dividing by $\int_M \chi  \ti{\omega}^m(0)\geq C^{-1}\delta'^{2n}e^{-(m-n)T}$  gives
$$(d_{g(T)}(p,q)+C\delta' )^2\geq (d_{\rm can}(f(p),f(q))-C\delta')^2-C\ov{\tau}\delta'^{-2n-2},$$
and taking $\ov{\tau}\leq C^{-1}\delta'^{2n+4}$ we obtain
\begin{equation}\label{final}
d_{g(T)}(p,q)\geq d_{\rm can}(f(p),f(q))- C\delta'.
\end{equation}
This fixes our choice of $\ov{\tau}$, and hence of $\eta$, which finally also fixes how large $T$ has to be. In summary, we have shown that \eqref{final} holds for sufficiently large $T$, and this finally concludes the proof of Claim 4, and hence of Proposition \ref{claim4}.
\end{proof}

\begin{rmk}
There is only one point in the proof of Theorem \ref{krf_gh} where it was essential to use estimate \eqref{quasiisom} (which is where we use the assumption that $N$ is smooth and $D$ is snc), which is to prove \eqref{ggg4}.  In the proof of \eqref{ggg4} we had to deal with the rather artificial possibility that the minimal $\mathcal{L}$-geodesic $\gamma$ there wanders in and out of the neighborhood $\ti{U}_{\eta/2}$ an unbounded number of times (what we called ``events'' in the proof). Here we want to remark that if one can find such $\gamma$ such that the number of events is bounded above by a uniform constant $A$, then one can prove \eqref{ggg4} (and hence Theorem \ref{krf_gh}) dropping the snc assumption on $D^{(1)}$. Indeed, for each event as above, we can estimate the $d_{\rm can}$-distance between the entry point $P:=f(\gamma(\tau_{\rm entry}))$ and the exit point $Q:=f(\gamma(\tau_{\rm exit}))$ by using that on $N\backslash U_{\eta/2}$
$$d_{\rm can}(P,Q)\leq Cd_{g_N}(P,Q)^\alpha,$$
for some uniform $\alpha>0$, by passing \eqref{4} to the limit. Since
$$d_{g_N}(P,Q)\leq \int_{\tau_{\rm entry}}^{\tau_{\rm exit}}|\de_\tau f(\gamma)|_{g_N}d\tau,$$
we see that
we can join $P$ and $Q$ with a path whose $g_{\rm can}$ length is at most
$$C\left(\int_{\tau_{\rm entry}}^{\tau_{\rm exit}}|\de_\tau f(\gamma)|_{g_N}d\tau\right)^\alpha\leq
C\left(\int_{\tau_{\rm entry}}^{\tau_{\rm exit}}|\de_\tau \gamma|_{\ti{g}(\tau)}d\tau\right)^\alpha,$$
and using this path to replace the portion of  $f(\gamma)$ with $\tau_{\rm entry}\leq \tau\leq\tau_{\rm exit},$ and repeating this for all the $A$ events, we obtain a new path $\gamma'$ joining $f(p)$ and $f(q)$ for which we have
\[\begin{split}
L_{g_{\rm can}}(\gamma')&\leq \int_J|\de_\tau f(\gamma)|_{g_{\rm can}}d\tau+ C\sum_{i=1}^A \left(\int_{\tau_{{\rm entry},i}}^{\tau_{{\rm exit},i}}|\de_\tau \gamma|_{\ti{g}(\tau)}d\tau\right)^\alpha\\
&\leq \int_J |\partial_\tau \gamma|_{\ti{g}(\tau)}  d\tau+C\delta'+ CA^{1-\alpha}\left(\int_I|\de_\tau \gamma|_{\ti{g}(\tau)}d\tau\right)^\alpha \\
&\leq \int_0^{\ov{\tau}} |\partial_\tau \gamma|_{\ti{g}(\tau)}  d\tau+C\delta'+ CA^{1-\alpha}\eta^{\alpha\beta}\\
&\leq \int_0^{\ov{\tau}} |\partial_\tau \gamma|_{\ti{g}(\tau)}  d\tau+C\delta',
\end{split}\]
by choosing $\eta$ sufficiently small, which proves \eqref{ggg4}.
\end{rmk}

\end{document}